\documentclass[11pt,a4paper]{amsart}
\usepackage[utf8]{inputenc}
\usepackage{amsfonts}
\usepackage{amssymb}
\usepackage{amsmath}
\usepackage{mathtools}
\usepackage{amsthm}
\usepackage{textcomp}
\usepackage{bbm}
\newtheorem{theorem}{Theorem}[section]
\newtheorem{lemma}[theorem]{Lemma}
\newtheorem{definition}[theorem]{Definition}

\newtheorem{corollary}[theorem]{Corollary}

\newtheorem{remark}[theorem]{Remark}
\usepackage{float}
\usepackage{enumitem}
\usepackage{booktabs}
\usepackage{pdflscape}
\usepackage{caption}
\usepackage[]{subfigure}
\usepackage{extarrows}
\usepackage[colorlinks=true,urlcolor=blue,citecolor=blue,linkcolor=blue,linktocpage,pdfpagelabels,bookmarksnumbered,bookmarksopen]{hyperref}

\begin{document}
\pagestyle{plain}
\title{P\lowercase{hase field model for multi-material shape optimization of inextensible rods}}
\author[P.\,Dondl]{Patrick Dondl}
\author[A.\,Maione]{Alberto Maione}
\author[S.\,Wolff-Vorbeck]{Steve Wolff-Vorbeck}
\address[P.\, Dondl]{Contributing Author\newline\indent Abteilung f\"{u}r Angewandte Mathematik and \textit{liv}MatS Cluster of Excellence\newline\indent Albert-Ludwigs-Universit\"{a}t Freiburg \newline\indent Hermann-Herder-Strasse 10, 79104 Freiburg i. Br. -- Germany}
\email{patrick.dondl@mathematik.uni-freiburg.de}
\address[A.\, Maione]{Corresponding Author\newline\indent Centre de Recerca Matemàtica\newline\indent Edifici C, Campus Bellaterra, 08193 Bellaterra, Spain}
\email{amaione@crm.cat}
\address[S.\,Wolff-Vorbeck]{Contributing Author\newline\indent Abteilung f\"{u}r Angewandte Mathematik\newline\indent Albert-Ludwigs-
Universit\"{a}t Freiburg \newline\indent Hermann-Herder-Strasse 10, 79104 Freiburg i. Br. -- Germany}
\thanks{P. Dondl and A. Maione are supported by the DFG SPP 2256 project. 
P. Dondl and S. Wolff-Vorbeck acknowledge the support of the DFG SPP 1886 project (DFG Projects number 422730790).
A. Maione is also supported by the INdAM-GNAMPA projects ``Equazioni differenziali alle derivate parziali di tipo misto o dipendenti da campi di vettori'' (Project number CUP\_E53C22001930001) and ``Pairing e div-curl lemma: estensioni a campi debolmente derivabili e differenziazione non locale'' (Project number CUP\_E53C23001670001).}
\begin{abstract}
We derive a model for the optimization of the bending and torsional rigidities of non-homogeneous elastic rods.
This is achieved by studying a sharp interface shape optimization problem with perimeter penalization, that treats both rigidities as objectives.
We then formulate a phase field approximation of the optimization problem and show the convergence to the aforementioned sharp interface model via $\Gamma$-convergence.
In the final part of this work we numerically approximate minimizers of the phase field problem by using a steepest descent approach and relate the resulting optimal shapes to the development of the morphology of plant stems.
\end{abstract}
\maketitle


\section{Introduction}


Tailoring resistance of a rod against bending and torsional deformations is a decisive factor in several fields of civil engineering and bioengineering, as well as in the development of plant morphology~\cite{ecsedi2005bounds,kim2000topology,niklas1992plant,Vogel1992,vogel2007living,wolff2022charting}.
In particular, in the construction of building components, it is necessary to optimize certain responses of rods subject to bending and torsional moments.  
A well-known way to achieve this is to optimize their bending and torsional rigidities \cite{kim2000topology,wolff2019twist}, which are appropriate measures of resistance to bending and torsional deformations.

A historical challenge is to find a rigorous mathematical formulation of these rigidities from a one-dimensional rod model, that can be seen as the limit of three-dimensional elastic energies, as the thickness of the rod tends to zero.
Classical methods in such derivation have a long history and are based on the so-called \textit{dimension reduction} method.
For a comprehensive introduction on the topic, we refer the interested reader to the monographs of Antman \cite{antman1973theory,antman2005problems} and to~\cite{kirchhoff1850gleichgewicht,villaggio1997mathematical}, for further discussions about the history of the subject.

In 2002, Mora and M\"uller \cite{mora2002derivation} were the first to rigorously answer this question.
They showed that the nonlinear bending-torsion theory can be obtained from the three-dimensional nonlinear elasticity, by studying the asymptotic behaviour (as $h$ tends to zero) of the sequence of energies
\begin{align}\label{energy0}
    \int\limits_{\Omega_h} W(\nabla v(x))\mathrm{d}x\,,\quad v \in W^{1,2}(\Omega_{h};\mathbb{R}^{3}),
\end{align}
by means of the De Giorgi and Franzoni $\Gamma$-convergence \cite{de1975tipo}.

 In this model, the rod is a three-dimensional set $\Omega_h=(0,L)\times hS$, where $L>0$, $S$ is the cross-section of the rod, that is a two-dimensional open, bounded and connected set with Lipschitz boundary, and $h\in\mathbb{R}^+$ is a small positive scaling parameter, that we refer to as the thickness of the rod.
In the subcase of isotropic materials, which is what we are concerned with below, the stored energy function $W \colon \mathbb{M}^{3\times 3} \rightarrow [0,+\infty]$ is required to satisfy the following standard assumptions:
\begin{itemize}
    \item [1.] $W \in \mathbf{C}^{0}(\mathbb{M}^{3\times 3})$ and $W$ is of class $\mathbf{C}^{2}$ in a neighbourhood of $SO(3)$;
    \item[2.] $W$ is \textit{frame-indifferent}, i.e. $W(A)$=$W(RA)$ for all $A \in \mathbb{M}^{3\times 3}$ and for all $R\in SO(3)$;
    \item[3.] $W(A)=0$ if $A\in SO(3)$ and there exists a constant $C>0$ such that
    \begin{align*}
    W(A)\geq C \operatorname{dist}^{2}(A,SO(3))\quad\text{for any }A\in\mathbb{M}^{3\times 3}\,;
    \end{align*}
    \item[4.] $W(A)=W(AR)$ for all $A \in \mathbb{M}^{3\times 3}$ and for all $R\in SO(3)$.
\end{itemize}
$\mathbb{M}^{3\times 3}$ and $SO(3)$ denote, respectively, the space of matrices of dimension $3$ and the group of all rotations about the origin of the three-dimensional Euclidean space $\mathbb{R}^{3}$ under the operation of composition.
As pointed out in \cite{acerbi1991variational}, the energies \eqref{energy0} scaled by $h^{2}$ correspond to stretching and shearing deformations, leading to a string theory.
In contrast, the energies \eqref{energy0} scaled by $h^{4}$ correspond to bending and torsional deformations, which leave the domain unextended and, thus, lead to a rod theory.

We emphasize that the authors of \cite{mora2002derivation} only took into account the case of inextensible \textit{homogeneous} rods, where the cross-sections $S$ contains only one single material.
This model has been later generalized by Neukamm~\cite{neukamm2011rigorous} to the case of \textit{non-homogeneous} rods, meaning that the stored energy function $W$ may also depend on the position of the point $x$ in $\Omega_h$.
Neukamm also takes into account the vertically and periodic distributed heterogeneities of the material, represented by a small parameter $\epsilon$, which leads $W=W_\epsilon(x,\nabla v(x))$ to depend also on $\epsilon$, and the nonlinear elastic energy to depend on the two scaling parameters $h$ and $\epsilon$.

The model presented in this paper is a special case on the one considered in \cite{neukamm2011rigorous}, which concerns non-oscillatory isotropic materials in the homogenized bending-torsion theory.
In what follows, we briefly describe the main ideas on how to derive the mathematical model.

 In the theories of pure bending and the Saint-Venant's theory of pure torsion for isotropic non-homogeneous rods, the bending rigidity and the torsional rigidity each depend on a single material constant (see e.g.~\cite{crandall1978introduction,ecsedi2005bounds,lekhnitskii1971torsion}).
These are, respectively, the space-dependent shear modulus $\mu(x_2,x_3)$ and the space-dependent Young's modulus $E(x_2,x_3)$, with $(x_2,x_3)\in S$.

In our model, we consider $\Omega_h$ fixed at $x_1=0$ and that the bending of $\Omega_h$ is due to bending moments $M_{x_2}$ and $M_{x_3}$, while the torsion to a torsional moment $T$ at $x_1=L$, see Fig.\ref{fig:1}.
\begin{figure}[H]
\centering
{\includegraphics[width=0.45\textwidth]{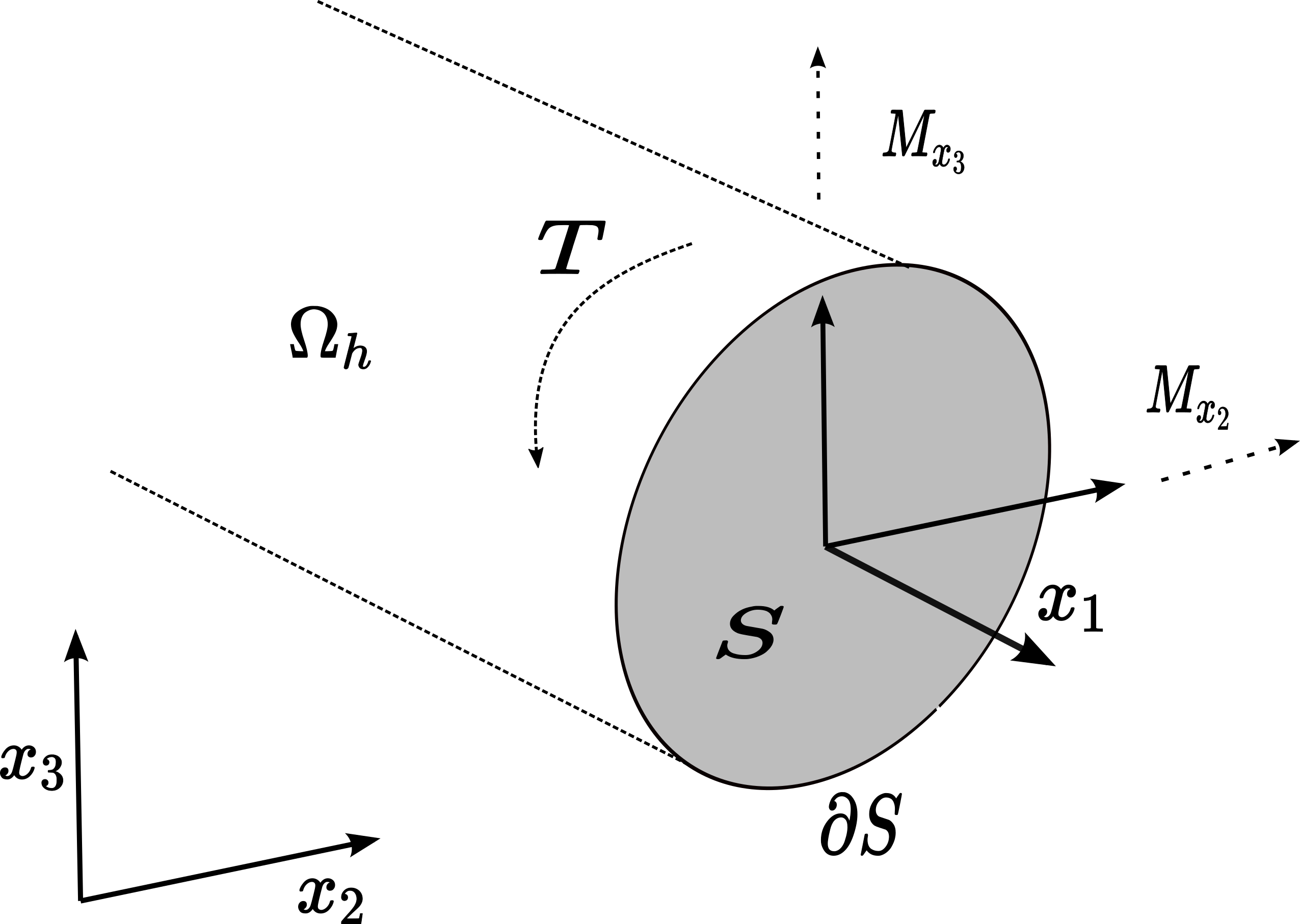}} \hspace{2em}
\caption[]{Inextensible rod subject to bending moments $M_{x_2},M_{x_3}$ and torsional moment $T$.}
\label{fig:1}
\end{figure}
Then, the torsional rigidity $D_T$, that is, the coefficient relating the squared torsion (per unit length) of a rod to it's elastic energy (per unit length),  can be determined by 
\begin{align*}
 D_{T}&\coloneqq \int\limits_{S}\mu(x_2,x_3)(x_2^2+x_3^2-x_2\partial_{x_3}w+x_3\partial_{x_2}w)~\mathrm{d}x_2\mathrm{d}x_3,  
\end{align*} 
where $w$ is the torsion function associated with the cross-section $S$, see~\eqref{eq:warp}. This torsion function describes the equilibrium out-of plane displacement of a twisted rod cross-section, see e.g. \cite[Chapter 9]{sadd2009elasticity}.

Moreover, by using the second moments of inertia and the product of inertia, respectively given by
\begin{align*}
    D_{x_2}=\int\limits_{S}E(x_2,x_3)&x_2^{2}~ \mathrm{d}x_2\mathrm{d}x_3,\quad
    D_{x_3}=\int\limits_{S}E(x_2,x_3)x_3^{2} ~ \mathrm{d}x_2\mathrm{d}x_3,\\
    D_{x_2x_3}&=\int\limits_{S}E(x_2,x_3)x_2x_3~ \mathrm{d}x_2\mathrm{d}x_3,
\end{align*}
we can determine the maximum bending and minimum bending rigidities $D_{\textrm{max}}$ and $D_{\textrm{min}}$ along the principal axes, as
\begin{equation*}
D_{\textrm{max/min}} = \frac{D_{x_2}+D_{x_3}}{2} \pm \sqrt{\frac{(D_{x_2}-D_{x_3})^{2}}{4}+(D_{x_2x_3})^{2}} =: D_{\textrm{mean}}\pm RM.
\end{equation*}
We note that the values $D_{\textrm{max/min}}$ are simply the two eigenvalues of the $2\times 2$ symmetric matrix with $D_{x_2}$, $D_{x_3}$ on the diagonal and $D_{x_2 x_3}$ on the off diagonal entries. This matrix relates the applied bending moment with the curvature of a rod, see \eqref{eq:lin}. 
The eigenvalues and eigenvectors of this matrix thus yield the maximal and minimal bending stiffness and the corresponding axes.

\newpage To allow the presence of \textit{multi-materials} inside $S$, we describe the space-dependent material constants $\mu$ and $E$ in terms of a scalar density function $u:S\to[0,1]$, which is assumed to be bounded from below by a positive constant $c$, that is
\begin{equation}\label{condu}
    u(x_2,x_3)\geq c>0\quad\text{for any }(x_2,x_3)\in S.
\end{equation}
By normalizing with respect to the moduli $\mu_\text{norm}$ and $E_\text{norm}$ of the stiffest material, we then replace $\mu$ and $E$, respectively, with the quantities
\begin{align*}
    \tilde\mu(x_2,x_3):= \mu_\text{norm} u(x_2,x_3) \quad\text{and} \quad\tilde E(x_2,x_3):= E_\text{norm} u(x_2,x_3).
\end{align*}
An optimization of the bending and torsional rigidities of the rod therefore consists of determining the optimal distributions and shapes of the different (competing) materials inside $S$.

The problem of finding optimal shapes and optimal topologies in structural mechanics has a long history.
Without aiming for completeness, we refer the interested reader to \cite{allaire1993numerical,bendsoe1995optimization,MR0820342,michell1904lviii,thomsen1992topology}, for an introduction on the subject, and to \cite{blank2016sharp,bourdin2003design,takezawa2010shape}, for recent approaches involving phase field methods, that are the ones that motivated the development of this work.

 Here, we are interested in determining the arrangements of two different materials inside the cross-section $S$, where the first material is softer and the other stiffer, in a way such that the bending and torsional rigidities of the rod are optimal.

This study is justified by observations on the morphology of plant stems, where the mechanical behaviour of the stem, subjected to bending and torsional loads, is mainly determined by two competing mechanically decisive materials distributed inside the stem cross-section~\cite{rowe2004diversity,wolff2022charting,wolff2021influence}.
In Fig. \ref{fig:condylo}, as an explanatory model, the cross-section of a liana of the type \textit{Condylocarpon Guianense}, including the stiffer secondary xylem (\textbf{1}) and the softer cortex (\textbf{2}), is illustrated.
\begin{figure}[H]
\centering
{\includegraphics[width=0.25\textwidth]{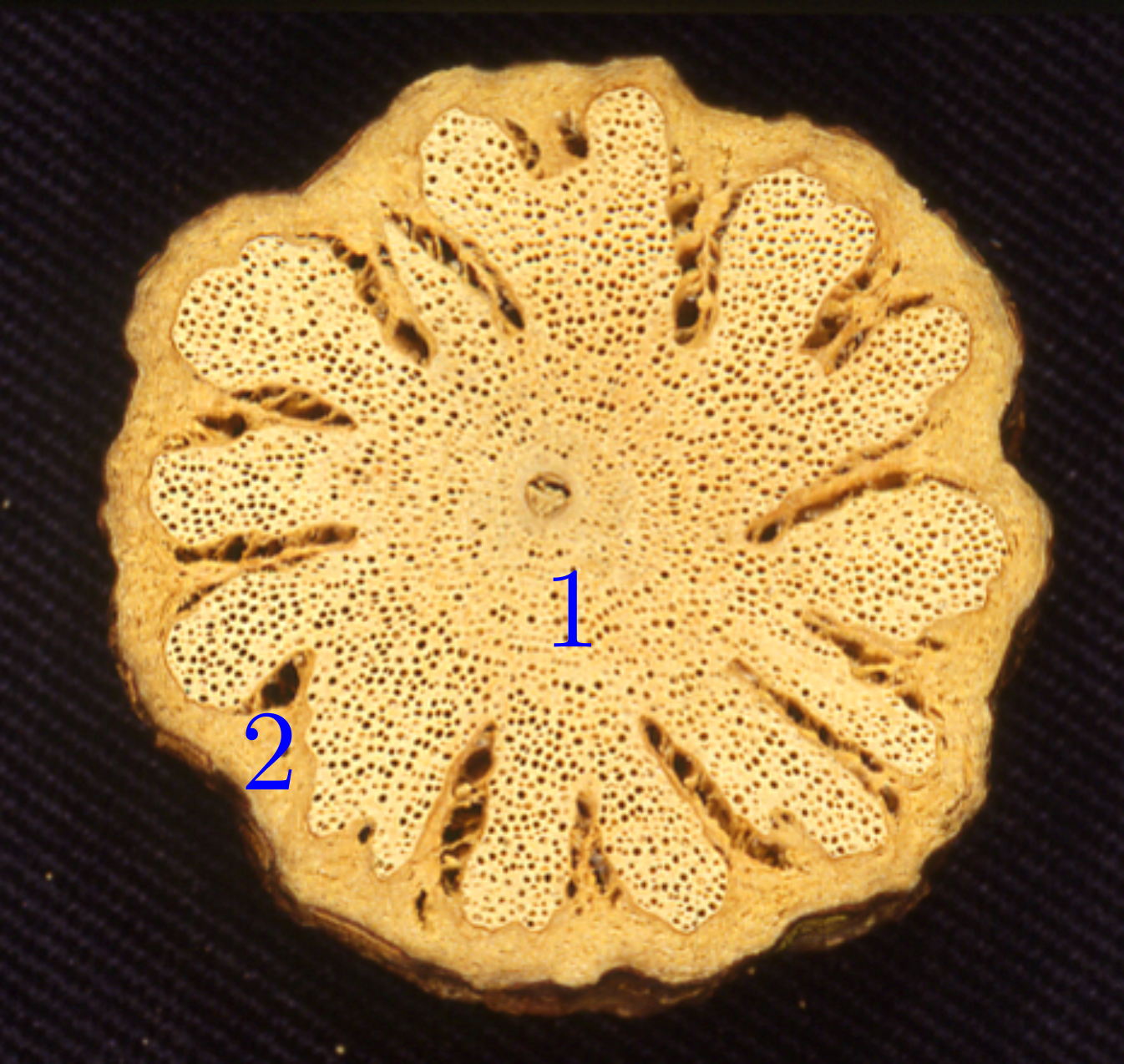}}
\caption[]{Cross-sections of liana \textit{Condylocarpon Guianense} in the non-self-supporting old ontogenetic stage after attachment to a support. The secondary xylem is marked with~(\textbf{1}) and the cortex is marked with~(\textbf{2}).~\textcopyright~ Plant Biomechanics Group Freiburg, annotated and used with permission.}
\label{fig:condylo}
\end{figure}
As pointed out in \cite{niklas1992plant,vogel2007living}, from an evolutionary point of view, plants are not inclined by a maximization of rigidity, but rather by an optimization of both strength and flexibility.
In general, a high ratio between the minimal bending rigidity $D_{\rm min}$ and the torsional rigidity $D_T$ can be observed.
This ratio is known in the literature as \textit{twist-to-bend ratio}~\cite{niklas1992plant,Vogel1992}, and a high ratio means high bending rigidity, compared to a lower torsional rigidity.

With this in mind, the shape optimization problem we will address consists of minimizing the following functional
\begin{align*}
     E_{\rm norm}\big(\sigma_1 D_{\rm mean}(u)+ \sigma_2 RM(u)\big)+ \sigma_3\mu_{\rm norm}D_{T}(u)+\gamma\operatorname{Per}(\{u=1\}),
\end{align*}
where the objectives are the mean bending rigidity $D_{\rm mean}$, the non-symmetric part $RM$, and the torsional rigidity $D_T$.
The competing regularizing term $\operatorname{Per}(\{u=1\})$ represents instead the perimeter of the regions inside the cross-section $S$ in which only the stiffer material is present, and the constants $\sigma_1,\sigma_2,\sigma_3, \gamma \in \mathbb{R}$ are weighting factors depending on the physical problem.
In the last section of the paper, we show some possible choices of $\sigma_1,\sigma_2,\sigma_3$ and $\gamma$, in order to numerically compute the solutions of the optimization problem via a phase field approach.

The paper is organized as follows:
in Section \ref{sec:mainTh}, we present an alternative approach, compared to Neukamm \cite{neukamm2011rigorous}, to derive the bending and torsional rigidities of an isotropic non-homogeneous elastic rod following the ideas of Mora and M\"uller \cite{mora2002derivation}, which consists in studying the asymptotic behaviour of a sequence of appropriate nonlinear elastic energies via $\Gamma$-convergence.

In Section \ref{sec:opt_rig}, we study the above optimization problem and provide, in Theorem \ref{thm:gamma_conv}, a sharp interface asymptotic of the phase field approach via $\Gamma$-convergence.
In Corollary \ref{lastcorollary}, through the standard technique derived from Modica and Mortola theory for phase transitions \cite{modica1987gradient}, we then recover the existence of solutions to the sharp interface optimization problem, by studying the limiting behaviour of solutions of the diffuse interface problems.

We conclude by studying, in Section~\ref{numerics}, the time discrete $L^2$-gradient flows of the approximating energies and, through a steepest descent approach, we present numerical results concerning the optimal distribution of the two materials inside $S$.
We then consider different optimization problems arising from different choices of the relative weightings factors and finally compare the results to the development of different morphologies in plant stems.
A graphical overview of the results is given in Fig. \ref{fig:num_sim}.


\section*{Notation}


Through the paper we will use the following notation: 
\begin{itemize}
    \item $\mathbb{M}^{3\times 3}_{\rm sym}$ and $\mathbb{M}^{3\times 3}_{\rm skew}$ denote, respectively, subspaces of $\mathbb{M}^{3\times 3}$ of symmetric and skew-symmetric matrices.
    \item For any $w\in W^{1,2}(\Omega)$, $y\in W^{1,2}(\Omega;\mathbb{R}^3)$ and $A=(A_1|A_2|A_3)\in\mathbb{M}^{3\times 3}$ we denote
        \begin{align*}
            \partial_{x_i} w&=\frac{\partial w}{\partial x_i}\quad\text{for any }i\in\{1,2,3\},\\
            y_{,i}&=(\partial_{x_i} y_1,\partial_{x_i} y_2, \partial_{x_i}y_3)\quad\text{for any }i\in\{1,2,3\},\\
            \nabla_h y&=\left(y_{,1}\left|\frac{1}{h} y_{,2} \right| \frac{1}{h}y_{,3}\right)\quad\text{for any }h\in\mathbb{R}^+,\\
            A_{,i}&=(A_{1,i},A_{2,i},A_{3,i})\quad\text{for any }i\in\{1,2,3\}.
        \end{align*}
\end{itemize}


\section{The mathematical model of non-homogeneous elastic rods}\label{sec:mainTh}


\subsection{Nonlinear bending-torsion theory and dimension reduction}


To study the bending and torsional rigidities of non-homogeneous rods we consider the following rescaled version of the $u$-dependent elastic energy~\eqref{energy0} 
\begin{equation*}
    I^{(h)}(y)\coloneqq\int\limits_{\Omega}u(x_2,x_3) W(\nabla_{h} y(x))\mathrm{d}x=\frac{1}{h^{2}}\int\limits_{\Omega_h}u\left(\frac{z_2}{h},\frac{z_3}{h}\right) W(\nabla v(z))\mathrm{d}z.
\end{equation*}
Here, to overcome the dependence of the rod $\Omega_h:=(0,L)\times hS$ on the scaling parameter $h\in\mathbb{R}^+$, we introduced the change of variables 
\[
(x_1,x_2,x_3):=\left(z_1,\frac{z_2}{h},\frac{z_3}{h}\right)\quad\text{for any }(z_1,z_2,z_3)\in\Omega_h,
\]
and denoted the new domain of integration by $\Omega:=(0,L)\times S$.
Moreover, we replaced $v$ and $u$, respectively, with $y\in W^{1,2}(\Omega;\mathbb{R}^3)$ and a not relabelled $u\in L^\infty(S)$, defined by
\begin{align*}
    y(x)&\coloneqq v(z(x))\\
    u(x_2,x_3)&\coloneqq u(z_2(x),z_3(x))
\end{align*}
for any $x\in\Omega$. In what follows, we assume the following centering condition
\begin{align*}
\int\limits_{S}u(x_2,x_3)x_2~\mathrm{d}x_2\mathrm{d}x_3 = \int\limits_{S}u(x_2,x_3)x_3~\mathrm{d}x_2\mathrm{d}x_3 &=0,
\end{align*}
which is satisfied by a proper choice of the coordinate system and is therefore not restrictive. 

As mentioned in the Introduction, the bending and torsional rigidities can be determined by an one-dimensional limit model, resulting from the study of the asymptotic behaviour, as $h$ goes to $0$, of a rescaled version of $I^{(h)}$ through the process of dimension reduction.
The mathematical tool that allows this study is the so-called De Giorgi and Franzoni $\Gamma$-convergence \cite{de1975tipo}.
A review of the literature on $\Gamma$-convergence is beyond the scope of this article.
We refer the interested reader to the in-depth monograph \cite{dal1993introduction} for a comprehensive introduction and deep analysis of the topic.

The $\Gamma$-convergence result that motivates and justifies the first part of this section is the following adaptation, to the case of multi-materials, of \cite[Theorem 3.1]{mora2002derivation} (which is also a special subcase of \cite[Theorem 3.1]{neukamm2011rigorous}).
It is based on the geometric rigidity result by Friesecke, James and M\"uller \cite{friesecke2002theorem}, which is essential to prove compactness.
A detailed proof of the Theorem \ref{mainTh} can be found in \cite[Section 6]{wolffthesis23}.

We first define the following subset of isometric deformations of $(0,L)$.
\begin{definition}
Let $(y,d_2,d_3)\in W^{2,2}(\Omega; \mathbb{R}^{3})\times W^{1,2}(\Omega;\mathbb{R}^{3})\times W^{1,2}(\Omega;\mathbb{R}^{3})$.
We say that $(y,d_2,d_3)$ belongs to the class $\mathcal{A}$ if at the same time the following three conditions hold:
\begin{itemize}
    \item $y$, $d_2$ and $d_3$ do not depend on $x_2$ and $x_3$,
    \item $|y_ {,1}|=|d_2|=|d_3|=1$,
    \item $y_{,1}\cdot d_{2} =  y_{,1}\cdot d_{3}= d_2\cdot d_{3}=0$.
\end{itemize}
We then define the following matrix-valued function
\begin{equation*}
R\coloneqq(y_{,1}\left|d_2\right|d_3)\in W^{2,2}((0,L); \mathbb{M}^{3\times 3}).
\end{equation*}
\end{definition}
Notice that $R$ is independent of $x_2$ and $x_3$, it belongs to the set $SO(3)$ and $R^{T}R_{,1}\in\mathbb{M}^{3\times 3}_{\rm skew}$, being
\begin{align*}
    (R^{T}R_{,1})_{kk}&=0\quad\text{for any }k\in\{1,2,3\},\\
    (R^{T}R_{,1})_{1k}&=-(R^{T}R_{,1})_{k1}=y_{,1}\cdot d_{k,1}\quad\text{for any }k\in\{2,3\},\\
    (R^{T}R_{,1})_{23}&=-(R^{T}R_{,1})_{32}=d_{2}\cdot d_{3,1}.
\end{align*}
We remark that the second relation above is related to the curvature caused by the bending moments, while the third one is related to the torsion of the rod, caused by torsional moments.
\begin{theorem}\label{mainTh}
Assume that the stored energy $W$ satisfies hypotheses $1.,2.,3.$ given in the Introduction, let $Q_3$ be twice the quadratic form of linearized elasticity
\begin{align*}
    Q_3(G)\coloneqq \frac{\partial^{2} W}{\partial F^{2}}(Id)(G,G)\quad\text{for any }G\in \mathbb{M}^{3\times 3}_{\rm sym},
\end{align*}
and denote $Q_2:\mathbb{M}^{3\times 3}_{\rm skew}\to[0,+\infty)$ the quadratic form defined through the minimization problem
\begin{align}\label{minprob1}
    Q_2(A)\coloneqq \min_{\alpha \in W^{1,2}(S;\mathbb{R}^{3})} \int_{S} u(x_2,x_3)\,Q_3\left(A\left(\begin{array}{rrr}0 \\x_{2} \\x_{3}\end{array}\right) \Big| \alpha_ {,2} \Big| \alpha_{,3}\right)\mathrm{d}x_2\mathrm{d}x_3
\end{align}
for any $A\in\mathbb{M}^{3\times 3}_{\rm skew}$, where the density function $u\in L^\infty(S)$ satisfies \eqref{condu}.

Then, there exists $I:W^{1,2}(\Omega;\mathbb{R}^{3})\times L^2(\Omega;\mathbb{R}^{3})\times L^2(\Omega;\mathbb{R}^{3})
\to\mathbb{R}\cup\{+\infty\}$ such that, up to subsequences,
\[
\left(\frac{1}{h^{2}}I^{(h)}\right)_h\quad\Gamma\text{-converges to }I,\text{ as }h\to0,
\]
in the strong and weak topologies of $W^{1,2}(\Omega;\mathbb{R}^{3})\times L^2(\Omega;\mathbb{R}^{3})\times L^2(\Omega;\mathbb{R}^{3})$, and the functional $I$ can be represented by
\begin{align*}
    I(y,d_2,d_3)=\left\{\begin{array}{ll} \frac{1}{2}\int\limits_{0}^{L} Q_2(R^{T}R_{,1})~\mathrm{d}x_1\quad &\text{if } (y,d_2,d_3)\in \mathcal{A}  \\
         +\infty &\text{elsewhere}\end{array}\right. 
\end{align*}
for any $(y,d_2,d_3)\in W^{1,2}(\Omega;\mathbb{R}^{3})\times L^2(\Omega;\mathbb{R}^{3})\times L^2(\Omega;\mathbb{R}^{3})$.
\end{theorem}

Note that the representation of the $\Gamma$-limit energy $I$ is well-posed, having the minimum problem~\eqref{minprob1} a unique solution $\alpha$, module a constant, which can be equivalently computed on the class of functions
\begin{align*}
    V\coloneqq{\Bigg\{}\alpha\in W^{1,2}(S;\mathbb{R}^{3}):\int_{S}u\alpha ~\mathrm{d}x_2\mathrm{d}x_3=\int_{S}u\nabla \alpha ~\mathrm{d}x_2\mathrm{d}x_3=0{\Bigg\}}\,.
\end{align*}
The existence of such a solution can be obtained by the direct method in the calculus of variations, noting that the functional we want to minimize is lower semicontinuous, with respect to the weak topology of $W^{1,2}(S;\mathbb{R}^3)$, and that the compactness is guaranteed by the fact that $Q_3$ is strictly positive definite on $\mathbb{M}^{3\times 3}_{\rm sym}$.
As for the uniqueness of the solution, it follows naturally from the strict convexity of $Q_3$ on $\mathbb{M}^{3\times 3}_{\rm sym}$, which is guaranteed by the assumption 3. on the stored energy function $W$.


\subsection{The case of isotropic materials}\label{sec:isotropic}
We now consider an application of Theorem~\ref{mainTh} to the case of \textit{isotropic materials} and derive the bending and torsional rigidities of non-homogeneous elastic rods, following the path of \cite{mora2002derivation}.
We then assume that the stored energy function $W$ also satisfies the additional isotropic assumption $4.$ given in the Introduction, so that $Q_3$ takes the well-known form
\begin{align}\label{iso_pc_ws}
    Q_3(G)=2\mu \left|\frac{G+G^{T}}{2}\right|^{2}+\lambda (\operatorname{trace}(G))^{2}\quad\text{for any }G\in \mathbb{M}^{3\times 3}_{\rm sym},
\end{align}
where $\mu$ and $\lambda$ are the so-called Lame's constants.

For any fixed skew-symmetric matrix $A=[a_{ij}]_{i,j=1,2,3}\in\mathbb{M}^{3\times 3}_{\rm skew}$
\[A=\begin{pmatrix}
0 & a_{12} & a_{13}\\
-a_{12} & 0 & a_{23}\\
-a_{13} & -a_{23} & 0
\end{pmatrix}\]
one can show, by the standard theory of PDEs, that the solution $\alpha=(\alpha_1,\alpha_2,\alpha_3)\in W^{1,2}(S;\mathbb{R}^3)$ of the minimization problem~\eqref{minprob1} satisfies, in the distributional sense, the following Euler-Lagrange equation
\begin{equation*}
    \displaystyle{\begin{cases}
    \operatorname{div}\Big(u\big[\partial_{x_2}\alpha_1+a_{23}x_3,\partial_{x_3}\alpha_1-a_{23}x_2\big]\Big) = 0 \quad&\text{in } S,\\
    \partial_{\theta}\alpha_1=-a_{23}(x_3,-x_2)\cdot \theta \quad&\text{on }\partial S,
    \end{cases}}
\end{equation*}
for what concerns the component $\alpha_1$, and the following system of Euler-Lagrange equations
\begin{equation*}
    \displaystyle{\begin{cases}
    \operatorname{div}( u[(2\mu+\lambda)\partial_{x_2}\alpha_2+\lambda(\partial_{x_3}\alpha_3+a_{12}x_2+a_{13}x_3),(\mu\partial_{x_3}\alpha_2+\mu\partial_{x_2}\alpha_3)])=0\\
    \operatorname{div}(u[(\mu\partial_{x_3}\alpha_2+\mu\partial_{x_2}\alpha_3),(2\mu+\lambda)\partial_{x_3}\alpha_3+\lambda(\partial_{x_2}\alpha_2+a_{13}x_3+a_{12}x_2)]) = 0
    \end{cases}}
\end{equation*}
in $S$, with boundary conditions
\begin{equation*}
    \displaystyle{\begin{cases}
    \big[(2\mu+\lambda)\partial_{x_2}\alpha_2+\lambda\partial_{x_3}\alpha_3,\mu\partial_{x_3}\alpha_2+\mu\partial_{x_2}\alpha_3\big]\cdot \theta =-\lambda(a_{12}x_2+a_{13}x_3)\cdot \theta_2\\
    \big[(\mu\partial_{x_3}\alpha_2+\mu\partial_{x_2}\alpha_3),(2\mu+\lambda)\partial_{x_3}\alpha_3+\lambda\partial_{x_2}\alpha_2\big]\cdot \theta = -\lambda(a_{12}x_2+a_{13}x_3)\cdot \theta_3
    \end{cases}}
\end{equation*}
on $\partial S$, where $\theta=(\theta_2,\theta_3)\in\mathbb{R}^2$ is the outer unit normal for the boundary $\partial S$ and $\partial_{\theta}$ denotes the normal derivative.
\begin{remark}
In the case of homogeneous materials, i.e. when the density $u$ is constant, the previous Euler-Lagrange equations lead to \cite[equations (3.19)-(3.20)]{mora2002derivation}, as expected.
\end{remark}
The solution to the previous Euler-Lagrange equations, belonging to the space $V$, is provided by $\alpha=(\alpha_1,\alpha_2,\alpha_3)$, whose components are
\begin{align*}
    \displaystyle{\begin{cases}\alpha_1(x_2,x_3)=a_{23}w(x_2,x_3),\\
    \alpha_2(x_2,x_3)= -\frac{1}{4}\frac{\lambda}{\mu+\lambda}(a_{12}x_2^2-a_{12}x_3^2+2a_{13}x_2x_3), \\ 
     \alpha_3(x_2,x_3)= -\frac{1}{4}\frac{\lambda}{\mu+\lambda}(-a_{13}x_2^2+a_{13}x_3^2+2a_{12}x_2x_3),
     \end{cases}}
\end{align*}
where $w:S\to\mathbb{R}$ is the torsion function on the non-homogeneous cross-section $S$, i.e. the function solving the Neumann problem
\begin{align}\label{eq:warp}
    \displaystyle{\begin{cases} \operatorname{div}\Big(u\big[\partial_{x_2}w+x_3,\partial_{x_3}w-x_2\big]\Big) = 0 \quad&\text{in }S,\\
    \partial_{\theta}w = -(x_3,-x_2)\cdot \theta \quad&\text{on }\partial S.
    \end{cases}}
\end{align}
By computing the value of the functional at these minimum points, we get that the corresponding value in $Q_2(A)$ is determined by
\begin{equation}\label{energy_new}
\begin{split}
    Q_2(A)&=\frac{\mu(3\lambda+2\mu)}{\lambda+\mu}\left(a_{12}^2\int\limits_{S}u(x_2,x_3)x_2^{2}\mathrm{d}x_2\mathrm{d}x_3+a_{13}^2\int\limits_{S}u(x_2,x_3)x_3^{2}\mathrm{d}x_2\mathrm{d}x_3\right)\\
    &\quad+\frac{\mu(3\lambda+2\mu)}{\lambda+\mu}\left(a_{12}a_{13}\int\limits_{S}u(x_2,x_3)x_2x_3\mathrm{d}x_2\mathrm{d}x_3\right)+\mu D_{T} a_{23}^{2},
\end{split}
\end{equation}
where the torsional rigidity $D_T$ is defined as
\begin{align}\label{tor_rgd}
    D_{T}=D_{T}(S)&\coloneqq \int\limits_{S}u(x_2,x_3)(x_2^2+x_3^2-x_2\partial_{x_3}w+x_3\partial_{x_2}w)\mathrm{d}x_2\mathrm{d}x_3. 
\end{align}

By considering the curvatures $a_{12}$ and $a_{13}$ in \eqref{energy_new}, we obtain the well-known moment curvature relation in bending theory (see e.g.~\cite{sadd2009elasticity})
\begin{align}\label{eq:lin}
\begin{pmatrix}
M_{ x_2 } \\
M_{ x_3 }
\end{pmatrix}
=  \frac{\mu(3\lambda+2\mu)}{\lambda+\mu} \begin{pmatrix}
D_{x_2}^{u}& D_{x_2x_3}^{u}\\
D_{x_2x_3}^{u}&D_{x_3}^{u}
\end{pmatrix}
\cdot \begin{pmatrix}
a_{12}\\
a_{13}
\end{pmatrix},
\end{align}
where $M_{x_2}$ and $M_{x_3}$ denote the bending moments applied at the end of the rod ($x_1=L$) and the $u$-dependent moments of inertia $D_{x_2}^{u}$ and $D_{x_3}^{u}$, as well as the product of inertia $D_{x_2x_3}^{u}$, are respectively given by
\begin{align*}
    D_{x_2}^{u}&=\int\limits_{S}u(x_2,x_3)x_2^{2}~ \mathrm{d}x_2\mathrm{d}x_3, \\
    D_{x_3}^{u}&=\int\limits_{S}u(x_2,x_3)x_3^{2} ~ \mathrm{d}x_2\mathrm{d}x_3, \\ 
    D_{x_2x_3}^{u}&= \int\limits_{S}u(x_2,x_3)x_2x_3~ \mathrm{d}x_2\mathrm{d}x_3.
\end{align*}
\begin{remark}\label{DminDmax}
The maximum bending rigidity $D_{\rm max}$ and the minimum bending rigidity $D_{\rm min}$ along the principal axes are determined by the maximal and minimal eigenvalues of the matrix in~\eqref{eq:lin}, leading to
\begin{equation}\label{eq:bend_rig}
D_{\textrm{max/min}}(u)= D_{\textrm{mean}}(u)\pm RM(u).
\end{equation}
Here, 
\begin{align*}
    D_{\textrm{mean}}(u) &= \frac{D^{u}_{x_2}+D^{u}_{x_3}}{2},\\  RM(u)&=\sqrt{\frac{(D^{u}_{x_2}-D^{u}_{x_3})^{2}}{4}+(D^{u}_{x_2x_3})^{2}}.
\end{align*}
\end{remark}


\subsection{Derivation of the torsional rigidity by stress functions}\label{prandtlsubsect}
The aim of the last part of this section is to obtain an equivalent formulation for the torsional rigidity $D_T$~\eqref{tor_rgd} by means of the \textit{Prandtl's stress function}.
It is a common trick in the literature, and allows three unknown stresses to be reduced to a single unknown stress function.

In this formulation, the shear stress components $\tau_{x_1x_2}$ and $\tau_{x_1x_3}$ of the Cauchy stress tensor are described by the derivatives of the auxiliary (scalar) function $\Phi:S\to\mathbb{R}$, through
\[
\tau_{x_1x_2}=\partial_{x_3}\Phi(x_2,x_3),\quad\tau_{x_1x_3}=-\partial_{x_2}\Phi(x_2,x_3),
\]
meaning that the Prandtl's stress function $\Phi$ is determined by the conditions
\begin{equation}\label{def_prandtl}
\begin{split}
    \frac{1}{u}\partial_{x_3}\Phi(x_2,x_3)&= \partial_{x_2}w(x_2,x_3)+x_3,\\
    \frac{1}{u}\partial_{x_2}\Phi(x_2,x_3)&= - (\partial_{x_3}w(x_2,x_3)-x_2),
\end{split}
\end{equation}
where $w$ is the torsion function introduced in \eqref{eq:warp}.

The following representation via Prandtl's stress functions will be a key point in the following sections.
It allows the Neumann boundary value problem \eqref{eq:warp} to be transformed into the equivalent problem with Dirichlet boundary conditions \eqref{boundvalue}, which is much easier to handle numerically. 

We assume that the cross-section $S$ has an outer boundary $\Gamma_0$ and finitely many (connected) inner boundary components $\Gamma_i$, $i\in\{1,...,J\}$, such that
 \[
 \bigcup_{i=0}^J\Gamma_i=\partial S.
 \]
Fix $i\in\{1,...,J\}$.
Then, by parametrizing any boundary component $\Gamma_i$ with $s\rightarrow (x_2(s),x_3(s))$, the corresponding components of the outer unit normal $\zeta=(\zeta_2,\zeta_3)\in\mathbb{R}^2$ for the boundary curve $\Gamma_i$ can be expressed as
\begin{align*}
	\zeta_2=\frac{d x_3}{ds},\quad
	\zeta_3=-\frac{d x_2}{ds},
\end{align*}
and, by the boundary condition on the Neumann problem \eqref{eq:warp} and by \eqref{def_prandtl}, we get
\begin{align*}
    \frac{d \Phi}{ds}=\frac{\partial \Phi}{\partial x_2}\frac{d x_2}{ds} +  \frac{\partial \Phi}{\partial x_3}\frac{d x_3}{ds} = 0\quad\text{on }\partial S.
\end{align*}
Therefore, for any connected component $\Gamma_i$ of the boundary $\partial S$, there exists a constant $\Phi_i\in\mathbb{R}$ such that $\Phi=\Phi_i$ on $\Gamma_i$.

By the transformation \eqref{def_prandtl}, the Neumann boundary value problem \eqref{eq:warp} can be reformulated as the Dirichlet boundary value problem
\begin{align}\label{boundvalue}
    \begin{cases}
    -\operatorname{div}\left(\frac{1}{u} \nabla\Phi\right) =2&\text{in }S, \\
    \Phi = \Phi_i &\text{on }\Gamma_i.
    \end{cases}
\end{align}

\begin{remark}
Note that we have assumed that $S$ is connected, as this is quite natural when working with rods.
For simply connected domains $S$, we also assume that $\Phi$ must be zero on $\partial S$ while, for multiply connected domains, that the value of $\Phi$ can be set to zero only on the outer boundary $\Gamma_{0}$ of $S$.  
\end{remark}

In \cite{sadd2009elasticity}, it has been pointed out that, in the inner boundaries $\Gamma_i$, the constants $\Phi_i$ are determined by the following condition 
\begin{align}\label{add_cond}
    \int\limits_{\Gamma_i} \frac{1}{u}\partial_{\theta}\Phi~\mathrm{d}s = 2|S_i| \quad\text{for any }i\in J,
\end{align}
where $\theta$ is the outer unit normal for $\partial S$ and $S_i$ are the sets enclosed by $\Gamma_i$.

Therefore, the torsional rigidity $D_{T}$ \eqref{tor_rgd} can be equivalently obtained by integrating the Prandtl's stress function $\Phi$ in $S$, leading to
\begin{align}\label{rig:prdt}
    D_{T}= 2\int\limits_{S}\Phi(x_2,x_3)~\mathrm{d}x_2\mathrm{d}x_3+ 2\sum\limits_{i=1}^{J}\Phi_i|S_i|. 
\end{align}
For more details about the representation \eqref{rig:prdt}, we refer the interested reader to~\cite{sadd2009elasticity}.

To derive a distributional formulation of the equation in \eqref{boundvalue}, we define the set $\tilde{S}$ as the set enclosed by the outer boundary $\Gamma_0$ of the cross-section $S$ (i.e. $S$ with any hole filled in, in case $S$ was not simply connected), and consider the space of admissible functions
\[
\mathcal{H}\coloneqq\{w\in W^{1,2}_0(\tilde{S})\colon\nabla w=0\text{ in }\tilde{S}\setminus S\}.
\]
By multiplying the equation in \eqref{boundvalue} by a test function $w\in \mathcal{H}$ and integrating over $S$, we obtain by partial integration 
\begin{align*}
    2\int\limits_{S}w~\mathrm{d}x_2\mathrm{d}x_3 &=  -\int\limits_{S} \operatorname{div}\left(\frac{1}{u} \nabla\Phi\right) w~\mathrm{d}x_2\mathrm{d}x_3\\
    &=\int\limits_{S}\frac{1}{u}\nabla \Phi \cdot \nabla w~\mathrm{d}x_2\mathrm{d}x_3 - \sum\limits_{i=1}^{J}\Bigg(\int\limits_{\Gamma_i}\Big(\frac{1}{u}\partial_{\theta}\Phi\Big) w ~\mathrm{d}s\Bigg),
\end{align*}
and, applying Gauss's theorem (with respect to $S_i$) on the last term on the r.h.s. above, and recalling that $\nabla w=0$ on $\tilde{S}\setminus S$, we find by~\eqref{add_cond}
\begin{align*}
  \int\limits_{\Gamma_i}\Big(\frac{1}{u}\partial_{\theta}\Phi\Big) w ~\mathrm{d}s= 2 w_i |S_i|
\end{align*}
for constants $w_i$, satisfying $w=w_i$ a.e. in $S_i$.

 Therefore, the distributional formulation of \eqref{boundvalue} reads
\begin{align}\label{eq:prandtl_phi}
   \int\limits_{\tilde{S}}\frac{1}{u}\nabla \Phi \cdot \nabla w~\mathrm{d}x_2\mathrm{d}x_3=2\int\limits_{\tilde{S}}w~\mathrm{d}x_2\mathrm{d}x_3 \quad\text{for all }w\in \mathcal{H}.
\end{align}

By Riesz's representation theorem, equation~\eqref{eq:prandtl_phi} admits a unique solution in $\mathcal{H}$ and the previous representation \eqref{rig:prdt} of the torsional rigidity $D_{T}$ can be written in a more compact way, by means of the larger set $\tilde{S}$, as
\begin{equation}\label{Dt}
    D_{T}= 2\int\limits_{S}\Phi(x_2,x_3)~\mathrm{d}x_2\mathrm{d}x_3+ 2\sum\limits_{i=1}^{N}\Phi_i|S_i|= 2\int\limits_{\tilde{S}} \Phi(x_2,x_3)~\mathrm{d}x_2\mathrm{d}x_3.
\end{equation}
In the following Section \ref{sec:opt_rig}, we will frequently use in our phase field approach this last representation of $D_T$.


\section{Optimization of the bending and torsional rigidities}\label{sec:opt_rig}


Inspired by the recent phase field approaches for structural topology optimization \cite{blank2016sharp,bourdin2003design}, we now study shape optimization problems for bending and torsion of inextensible non-homogeneous elastic rods involving two different isotropic materials, a stiffer and a softer, within a fixed cross-section $S$.

In what follows, the distribution of the materials inside $S$ is described by a function $\varphi$ belonging to the class
\[
\mathcal{U} := \Bigg\{ \varphi \in  BV(S;\{0,1\}) \colon \int\limits_{S} \varphi(x_2,x_3)~\mathrm{d}x_2\mathrm{d}x_3 = m_1|S|,\text{ with } m_1\in (0,1)\Bigg\},
\]
where the volume constraint $m_1|S|$ is given by means of the mass of the stiffer material $m_1$.

An equivalent description of such distribution is possible by considering the one-to-one corresponding function $u^\varphi\in BV(S;\{c,1\})$, defined by
\begin{align}\label{eq:u_phi}
  u^\varphi(x_2,x_3)=\varphi(x_2,x_3)(1-c)+c\quad\text{for any }(x_2,x_3)\in S,
\end{align}
for a given positive constant $c$.
This last representation is intended to make the notation more familiar, linking it to what was introduced in the previous section.

The set $\{u^\varphi=1\}$, or equivalently $\{\varphi=1\}$, then describes regions where only the stiffer material is present (up to Lebesgue measure zero sets), while the set $\{u^\varphi=c\}$ represents regions containing only the softer material.

 If $c=1$, we are in the special case of a single homogeneous material and
\[
u^\varphi(x_2,x_3)=1\quad\text{for any }(x_2,x_3)\in S.
\]

Denote $\mu_{\rm norm}$ and $\lambda_{\rm norm}$ the Lam\'e constants of the stiffer material. Then, for any $(x_2,x_3)\in S$, the Lam\'e constants $\mu$ and $\lambda$ become
\begin{align*}
    \mu(x_2,x_3)&=\mu_{\rm norm} u^{\varphi}(x_2,x_3),\\
    \lambda(x_2,x_3)&=\lambda_{\rm norm} u^{\varphi}(x_2,x_3).
\end{align*}

For ease of reading, this section is divided into three parts: in Section \ref{sub3.1}, we introduce the perimeter penalized shape optimization problem \eqref{min:sharp}. We are interested in proving the existence of solutions for \eqref{min:sharp}, by studying a sharp interface limit in the sense of Modica and Mortola \cite{modica1987gradient} (see Section \ref{sub3.3}).
This is done via $\Gamma$-convergence \cite{de1975tipo}, applied to the sequence of approximating problems \eqref{min:diff}, which will be introduced in Section \ref{sub3.2}.


\subsection{Perimeter penalized shape optimization}\label{sub3.1}

We aim to show the existence of solutions to the following minimization problem 
\begin{align}\label{min:sharp}
    \min_{\varphi \in \mathcal{U}} J_0(\varphi),
\end{align}
where the functional $J_0\colon \mathcal{U} \rightarrow \mathbb{R}\cup\{+\infty\}$ is defined by
\begin{equation}\label{J0}
\begin{split}
  J_0(\varphi)&\coloneqq \sigma_1 D_{\rm mean}(\varphi)+\sigma_2 RM(\varphi)+ \sigma_3\mu_{\rm norm}D_{T}(\varphi)\\
  &\quad+\gamma \big( \operatorname{Per}(\{\varphi=1\}) + \int\limits_{\partial S} \varphi_{|_{\partial S}}~\mathrm{d}\mathcal{H}^{1}\big).
\end{split}
\end{equation}
The first two objectives in the optimization problem \eqref{min:sharp} are the average bending rigidity
\begin{equation}\label{Dmean}
   D_{\textrm{mean}}(\varphi) = \frac{\mu_{\rm norm}(3\lambda_{\rm norm}+2\mu_{\rm norm})}{\lambda_{\rm norm}+\mu_{\rm norm}} \Bigg(\frac{D_{x_2}^{u^\varphi}+D_{x_3}^{u^\varphi}}{2}\Bigg)
\end{equation}
and the non-symmetric part 
\begin{equation}\label{RM}
    RM(\varphi)=\frac{\mu_{\rm norm}(3\lambda_{\rm norm}+2\mu_{\rm norm})}{\lambda_{\rm norm}+\mu_{\rm norm}} \sqrt{\frac{(D^{u^\varphi}_{x_2}- D^{u^\varphi}_{x_3})^{2}}{4}+(D^{u^\varphi}_{x_2x_3})^{2}}\,.
\end{equation}
The remaining objective in the optimization problem \eqref{min:sharp} is the torsional rigidity $D_{T}$ represented, as in \eqref{Dt}, via Prandtl's stress functions.

The last terms that appear in the definition of $J_0$ \eqref{J0} are the regularizing term $\operatorname{Per}(\{\varphi=1\})$, which is the perimeter of the regions inside $S$ where only the stiffer material is present, the trace of $\varphi$ in the sense of \cite[Theorem 3.87]{ambrosio2000functions}, $\varphi_{|_{\partial S}}$, and $\sigma_1,\sigma_2,\sigma_3,\gamma \in \mathbb{R}$, which are weighting factors.

 With different choices of $\sigma_1$ and $\sigma_2$, we can optimize the minimum bending rigidity $D_{\rm min}$ (when $\sigma_1=-\sigma_2$), or the maximum bending rigidity $D_{\rm max}$ (when $\sigma_1=\sigma_2$). We refer to Section \ref{numerics} for further examples of optimizations. 

The moments of inertia and the product of inertia introduced in \eqref{Dmean} and \eqref{RM}, namely
\begin{align*}
    D_{x_2}^{u^\varphi}=\int\limits_{S}u^\varphi \hat{x}_2^{2}~ \mathrm{d}x_2\mathrm{d}x_3,
    D_{x_3}^{u^\varphi}=\int\limits_{S}u^\varphi \hat{x}_3^{2} ~ \mathrm{d}x_2\mathrm{d}x_3,
    D_{x_2x_3}^{u^\varphi}= \int\limits_{S}u^\varphi \hat{x}_2\hat{x}_3~ \mathrm{d}x_2\mathrm{d}x_3,  
\end{align*}
refer to the following proper choice of the coordinate system $(\hat{x}_2,\hat{x}_3)$
\begin{align*}
    \hat{x}_2=x_2-\frac{\int_S\limits u^\varphi x_2~\mathrm{d}x_2x_3}{\int_S\limits u^\varphi~\mathrm{d}x_2\mathrm{d}x_3},\quad
    \hat{x}_3=x_3-\frac{\int_S\limits u^\varphi x_3~\mathrm{d}x_2x_3}{\int_S\limits u^\varphi~\mathrm{d}x_2\mathrm{d}x_3}\,.
\end{align*}
In fact, unlike Section~\ref{sec:mainTh} where the coordinate system was $(x_2,x_3)=(\hat{x}_2,\hat{x}_3)$, in this case we have to work with (arbitrary) coordinate systems, which involve additional terms resulting from the nature of the shape optimization problem.

From a pure mathematical point of view, solutions to problem \eqref{min:sharp} are equivalent to minimizers of the extended functional $\mathcal{J}_0:L^1(S)\to\mathbb{R}\cup\{+\infty\}$, defined by
\begin{align}\label{J0L1}
  \mathcal{J}_0(\varphi)  &:=
  \begin{cases}
      J_0(\varphi) &\quad\text{in }\mathcal{U}, \\
    +\infty &\quad\text{in }L^1(S)\setminus\mathcal{U}.
  \end{cases}
\end{align}
In the last result of this section, Corollary \ref{lastcorollary}, we will show that the existence of solutions to problem \eqref{min:sharp} can be obtained as an application of the fundamental theorem of $\Gamma$-convergence \cite[Theorem 7.8]{dal1993introduction}.
Similarly to the original case of Modica and Mortola \cite{modica1987gradient}, it will be crucial the presence of the regularizing term 
\begin{equation}\label{regterm}
    \operatorname{Per}(\{\varphi=1\})+ \int\limits_{\partial S} \varphi_{|_{\partial S}}~\mathrm{d}\mathcal{H}^{1}.
\end{equation}
The reason for this will be clarified in the next two subsections.


\subsection{The phase field approach}\label{sub3.2}

To solve numerically problem~\eqref{min:sharp} (see Section \ref{numerics} for details), it is convenient to replace and approximate both the set  $\mathcal{U}$ and the solutions $\varphi$, respectively, with the Sobolev space
\begin{align*}
  \mathcal{U}_{\rm ad} \coloneqq \Bigg\{ w \in W^{1,2}_0(S) \colon 0\leq w \leq 1, \int\limits_{S} w(x_2,x_3)~\mathrm{d}x_2\mathrm{d}x_3 = m_1|S| \Bigg\},
\end{align*}
with $m_1\in (0,1)$, which is the set of admissible functions in the phase field approach, and the solutions (depending on $\epsilon$) of the approximating problems
\begin{align}\label{min:diff}
    \min_{\varphi \in \mathcal{U}_{\rm ad}}J_\epsilon(\varphi)\,,\quad\epsilon>0\, , 
\end{align}
where $J_\epsilon\colon \mathcal{U}_{\rm ad} \rightarrow \mathbb{R}\cup\{+\infty\}$ is defined by
\begin{equation}\label{Jeps}
    J_\epsilon(\varphi)\coloneqq \sigma_1 D_{\rm mean}(\varphi)+ \sigma_2 RM(\varphi)+ \sigma_3\mu_{\rm norm}D_{T}(\varphi)+\frac{\gamma}{c_0}E_\epsilon(\varphi)\,.
\end{equation}
In \eqref{Jeps}, the regularizing term \eqref{regterm} of problem~\eqref{min:sharp}, is approximated by a sequence of Ginzburg-Landau energies $E_\epsilon:\mathcal{U}_{\rm ad}\to\mathbb{R}_0^+
\cup\{+\infty\}$, defined by
\begin{equation}\label{Eeps}
    E_\epsilon(\varphi):=\int\limits_{S} \left(\frac{\epsilon}{2}|\nabla \varphi|^{2} + \frac{1}{\epsilon} F(\varphi)\right) ~\mathrm{d}x_2\mathrm{d}x_3,
\end{equation}
where a prototype of the double obstacle potential $F:\mathcal{U}_{\rm ad}\to\mathbb{R}_0^+
\cup\{+\infty\}$ is
\begin{align}\label{double_well}
    F(\varphi)= \frac{1}{4}\varphi^2(1-\varphi)^2\quad\text{for any }\varphi\in\mathcal{U}_{\rm ad}\,.
\end{align}
\begin{remark}
In accordance with the classical literature (see e.g. \cite{modica1987gradient}), the presence of the normalizing constant $c_0$ in \eqref{Jeps} is intended to retrieve a normalized version of the regularizing term, passing to the limit.
\end{remark}
As in the previous case, the solutions to \eqref{min:diff} are critical points (minimizers) for the sequence of extended functionals $\mathcal{J}_\epsilon:L^1(S)\to\mathbb{R}\cup\{+\infty\}$, defined by
\begin{align}\label{JepsL1}
  \mathcal{J}_\epsilon(\varphi)  &\coloneqq
  \begin{cases}
      J_\epsilon(\varphi)\quad & \text{in } \mathcal{U_{\rm ad}}, \\
    +\infty\quad & \text{in }L^1(S)\setminus\mathcal{U_{\rm ad}}.
  \end{cases}
\end{align}

The first step in the phase field approach is to prove the regularity of the objectives in the optimization problem \eqref{min:diff} in the right topology, i.e. the continuity of $D_{\rm mean}$, $RM$ and $D_T$ in the strong topology of $L^1(S)$.

Later, in Theorem \ref{existenceminimizers}, we will show the existence of a solution to any minimization problem \eqref{min:diff}.
\begin{lemma}\label{continuity}
Let $D_{T}$, $D_{\rm mean}$ and $ RM$ be the objectives in problem~\eqref{min:sharp}, defined respectively in \eqref{Dt}, \eqref{Dmean} and \eqref{RM}.
Then, the functional
\begin{align*}
\sigma_1 D_{\rm mean}(\varphi)+\sigma_2RM(\varphi)+\mu_{\rm norm}\sigma_3 D_{T}(\varphi)
\end{align*}
for any $\varphi\in\Big\{w\in L^1(S):\,w\in[0,1]\text{ a.e. in }S\Big\}$ is continuous in the strong topology of $L^{1}(S)$, for any possible choice of $\sigma_1,\sigma_2,\sigma_3 \in \mathbb{R}$.
\end{lemma}
\begin{proof}
We first notice that the continuity of 
\[
\sigma_1 D_{\rm mean}(\varphi)+\sigma_2RM(\varphi)
\]
in the strong topology of $L^1(S)$ directly follows from the boundedness of the cross-section $S$, for any possible choice of the constants $\sigma_1$ and $\sigma_2$.
We then just focus on the torsional rigidity term $D_T$ and the Prandtl's stress function $\Phi\in\mathcal{H}$, associated to it in the sense of \eqref{Dt}.

Let us fix $\varphi_{k}$ and $\varphi_{0}$ in $L^{1}(S)$ (and the corresponding functions $u^{\varphi_{k}}$ and $u^\varphi_{0}$) and assume that the sequence $(\varphi_{k})_k$ strongly converges to $\varphi_{0}$ in $L^{1}(S)$, as $k \rightarrow\infty$.
Moreover, for any $k\in\mathbb{N}$ denote $\Phi_k\in\mathcal{H}$ the Prandtl's stress function associated with $\varphi_k$, in the sense of \eqref{eq:prandtl_phi} (with $u=u^{\varphi_k}$).

 To conclude, we show the existence of a limit Prandtl's stress function $\Phi_0$, which is still a solution of \eqref{eq:prandtl_phi} (with $u=u^{\varphi_0}$) and to which the sequence $(\Phi_k)_k$ strongly converges in $L^1(\tilde{S})$. 

By the Poincar\'e inequality, the sequence $(\Phi_k)_k$ is bounded in $W^{1,2}_0(\tilde{S})$ and, by reflexivity, there exist a subsequence of $(\Phi_k)_k$, denoted by $(\Phi_{k_j})_{k_j}$, and a limit $\Phi_0\in W^{1,2}_0(\tilde{S})$ such that $\Phi_{k_j}$ weakly converges to $\Phi_0$ in $W^{1,2}_0(\tilde{S})$. 

Note that also the sequence $\left(\frac{1}{u^{\varphi_{k_j}}}\Phi_{k_j}\right)_{k_j}$ is bounded in $W^{1,2}_0(\tilde{S})$ and, up to a further (not-relabeled) subsequence, the dominated convergence theorem ensures that
\begin{align*}
    \Bigg|\int\limits_{\tilde{S}} \frac{1}{u^{\varphi_{k_j}}}\nabla \Phi_{k_j} \cdot \nabla v~\mathrm{d}x_2\mathrm{d}x_3&- \int\limits_{\tilde{S}} \frac{1}{u^{\varphi_0}}   \nabla \Phi_0 \cdot  \nabla v~\mathrm{d}x_2\mathrm{d}x_3\Bigg| \\
    &\leq \int\limits_{\tilde{S}} \left|\frac{1}{u^{\varphi_{k_j}}}-\frac{1}{u^{\varphi_0}}\right||\nabla \Phi_{k_j} \cdot  \nabla v|~\mathrm{d}x_2\mathrm{d}x_3 \\
    &\quad+\int\limits_{\tilde{S}} \frac{1}{u^{\varphi_0}}\left|(\nabla \Phi_{k_j}-\nabla \Phi_0) \cdot  \nabla v\right|~\mathrm{d}x_2\mathrm{d}x_3\to 0
\end{align*}
as ${k_j}\to\infty$, for any $v \in \mathcal{H}$.
Thus, $\Phi_0$ is a solution to problem~\eqref{eq:prandtl_phi} (with $u=u^{\varphi_0}$), and belongs to the set $\mathcal{H}$.

By applying the previous argument to any subsequence of the starting sequence $(\Phi_k)_k$, we find that every subsequence of $(\Phi_k)_k$ has a subsequence weakly convergent in $\mathcal{H}$ to the unique solution of~\eqref{eq:prandtl_phi}. 

 Then, we get the weak convergence of the whole sequence $(\Phi_k)_k$ in $\mathcal{H}$ and, by Rellich theorem, the strong convergence to $\Phi_0$ in $L^{1}(\tilde{S})$ and the thesis readily follows. 
\end{proof}
With Lemma \ref{continuity} in hand, we are now in a position to prove the main result of this section, which ensures the existence of solutions to each problem \eqref{min:diff}.
\begin{theorem}\label{existenceminimizers}
There exists at least a solution to \eqref{min:diff}, for any $\epsilon>0$.
\end{theorem}
\begin{proof}
The existence of minimizers is shown by means of the direct method in the calculus of variations.

Fix $\epsilon>0$ and $\sigma_1,\sigma_2,\sigma_3\in\mathbb{R}$, consider the functional $J_\epsilon$ introduced in \eqref{Jeps} and, for any $\varphi\in\mathcal{U}_{\rm ad}$, let $\Phi\in\mathcal{H}$ be the Prandtl's stress function associated with $\varphi$, in the sense of \eqref{eq:prandtl_phi}.
Then, $\varphi$ is bounded on $\mathcal{U}_{\rm ad}$ and, by the boundedness of $S$, there exist positive constants $C_1,C_2$, only depending on $S,\sigma_1,\sigma_2$ and $\sigma_3$, such that 
\begin{align*}
    \sigma_1 D_{\rm mean}(\varphi)\pm \sigma_2 RM(\varphi)&> -C_1\quad\text{and}\\
    \sigma_3 D_{T}(\varphi)&> -C_2\quad\text{for any }\varphi\in\mathcal{U}_{\rm ad}.
\end{align*}
Therefore, since $F$ is a double obstacle potential, there is a control from below on $J_\epsilon$ by means of $\epsilon\|\varphi\|_{W^{1,2}_0(S)}$, which yields the coercivity of $J_\epsilon$ in $L^1(S)$.

Let $(\varphi_k)_{k}$ be a minimizing sequence of $J_\epsilon$, meaning that $(J_\epsilon(\varphi_k))_k$ converges to the infimum of $J_\epsilon$ in $\mathcal{U}_{\rm ad}$, as $k$ goes to $\infty$.
Then, $\sup_{k}J_\epsilon(\varphi_k) <\infty$ and so
\begin{align*}
    \sup_{k \in \mathbb{N}}\|\nabla \varphi_k\|_{L^2(S)} < \infty,
\end{align*}
that is, $(\varphi_k)_k$ is bounded in $W^{1,2}_0(S)$.

Therefore, by reflexivity, there exists $\varphi_{0}\in W^{1,2}_0(S)$ such that, up to a not relabelled subsequence, $(\varphi_k)_k$ converges to $\varphi_{0}$ weakly in $W^{1,2}_0(S)$, strongly in $L^{2}(S)$ and a.e. in $S$ (this last convergence is actually pointwise, being $(\varphi_k)_k$ a minimizing sequence).
By the pointwise convergence, $\varphi_0\in[0,1]$ and, as a consequence of the $L^2$-convergence, $\varphi_0$ satisfies the mass constraint
\[
\int\limits_{S} \varphi_0(x_2,x_3)~\mathrm{d}x_2\mathrm{d}x_3 =m_1|S|,
\]
so that $\varphi_0 \in \mathcal{U}_{\rm ad}$. 

Moreover, by the continuity of $F$ and the lower semicontinuity of the norm $\|\cdot\|_{W^{1,2}_0(S)}$, the Ginzburg-Landau energy $E_\epsilon$ is sequentially lower semicontinuous and, by Lemma~\ref{continuity}, the functional $J_\epsilon$ is lower semicontinuous too.

The thesis then follows as a consequence of the Weierstrass Theorem (direct method in the calculus of variations, see e.g. \cite[Theorem 1.15]{dal1993introduction}).
\end{proof}


\subsection{Sharp interface limit}\label{sub3.3}

We conclude this section by showing that the sequence of functionals $\mathcal{J}_\epsilon$, defined in \eqref{JepsL1},  $\Gamma$-converges to the functional $\mathcal{J}_0$, introduced in \eqref{J0L1}, in the strong topology of $L^1(S)$.
As a consequence of our Modica-Mortola-type Theorem \ref{thm:gamma_conv}, we will finally show in Corollary \ref{lastcorollary} the existence of a solution for the minimization problem \eqref{min:sharp}, as a consequence of the fundamental theorem of $\Gamma$-convergence \cite[Theorem 7.8]{dal1993introduction}.

Having defined the set of admissible functions $\mathcal{U}_{\rm ad}$ by imposing homogeneous Dirichlet boundary conditions on the phase field variable, we follow the ideas of~\cite{huttl2022phase} in order to prove the $\Gamma$-convergence Theorem \ref{thm:gamma_conv}.
They consist of constructing the recovery sequence in the spirit of Modica-Mortola~\cite{modica1987gradient} and Sternberg~\cite{sternberg1988effect}, and by using an additional cut-off procedure, as in~\cite{bourdin2003design}.

Note that the corresponding construction of the recovery sequences in~\cite{huttl2022phase} represents a special case of the fundamental procedure derived in~\cite{owen1990minimizers} for more general Dirichlet boundary conditions imposed on the phase field variables.

\begin{theorem}\label{thm:gamma_conv}
Let $\mathcal{J}_\epsilon,\mathcal{J}_0:L^1(S)\to\mathbb{R}\cup\{+\infty\}$, $\epsilon>0$, be the functionals defined in \eqref{JepsL1} and in \eqref{J0L1}, respectively. 
Then,
\[
(\mathcal{J}_\epsilon)_\epsilon\quad \Gamma\text{-converges to }\mathcal{J}_0
\]
in the strong topology of $L^1(S)$, as $\epsilon$ goes to $0$.
\end{theorem}
\begin{proof}
Denote $\mathbbm{1}_A$ the indicator function of any set $A\subset L^1(S)$, that is,
\[
\mathbbm{1}_A(x):=
\begin{cases}
    0&\text{ if } x\in A,\\
    +\infty&\text{ otherwise}.
\end{cases}
\]
Then, we can rewrite any functional $\mathcal{J}_\epsilon$ in \eqref{JepsL1} as
\[
\mathcal{J}_\epsilon(\varphi)=D_r(\varphi)+\frac{\gamma}{c_0}E_\epsilon(\varphi)+\mathbbm{1}_{K}\quad\text{for any }\varphi\in L^1(S)\text{ and }\epsilon>0,
\]
where
\begin{align*}
    D_r(\varphi)&\coloneqq\sigma_1 D_{\rm mean}(\varphi)+ \sigma_2 RM(\varphi)+\sigma_3\mu_{\rm norm} D_{T}(\varphi)\quad\text{for any }\varphi\in\mathcal{U}_{\rm ad},\\
    K &\coloneqq \Bigg\{ \varphi \in L^{1}(S)\colon \int\limits_{S} \varphi(x_2,x_3)~\mathrm{d}x_2\mathrm{d}x_3 = m_1|S| \Bigg\}.
\end{align*}

By \cite[Theorem 3.17]{huttl2022phase} and \cite[Theorem 2.1]{owen1990minimizers}, the sequence of energies $\mathcal{E}_\epsilon:L^1(S)\to\mathbb{R}\cup\{+\infty\}$, defined by
\begin{align*}
    \mathcal{E}_{\epsilon}(\varphi)&:=E_\epsilon(\varphi)+\mathbbm{1}_{W^{1,2}_0(S;[0,1])}= 
    \begin{cases}
        E_\epsilon(\varphi) &\quad\text{in }W^{1,2}_0(S;[0,1]),\\
    +\infty&\quad\text{in }L^1(S)\setminus W^{1,2}_0(S;[0,1]),
    \end{cases}
\end{align*}
$\Gamma$-converges in the strong topology of $L^{1}(S)$, as $\epsilon \rightarrow 0$, to the limit functional $\mathcal{E}_0:L^1(S)\to\mathbb{R}\cup\{+\infty\}$, which can be represented by
\begin{align*}
    \mathcal{E}_{0}(\varphi) \coloneqq
    \begin{cases}
        c_0\Big(\operatorname{Per}(\{\varphi=1\}) + \int\limits_{\partial S} \varphi_{|_{\partial S}}~\mathrm{d}\mathcal{H}^{1}\Big) &\quad\text{in }BV(\{0,1\}), \\
    +\infty&\quad\text{in }L^1(S)\setminus BV(\{0,1\}).
    \end{cases}
\end{align*}
Up to a further subsequence, the constructions by \cite[Theorem 4.8]{garcke2021phase} and \cite[Theorem 3.18]{huttl2022phase}, guarantee that 
\[
\left(\frac{\gamma}{c_0}\mathcal{E}_{\epsilon}(\varphi) + \mathbbm{1}_K\right)_\epsilon\quad\Gamma\text{-converges to}\quad\gamma\mathcal{E}_0+\mathbbm{1}_K\,,\quad\text{as }\epsilon\to 0,
\]
and the assertion then follows by Lemma~\ref{continuity} and \cite[Proposition 6.21]{dal1993introduction}.
\end{proof}
As a consequence of the previous result, we finally show that any sequence of minimizers $(\varphi_\epsilon)_\epsilon$ of $(\mathcal{J}_\epsilon)_\epsilon$ strongly converges in $L^1(S)$ to a solution of the sharp interface problem~\eqref{min:sharp}, as $\epsilon$ goes to $0$.
\begin{corollary}\label{lastcorollary}
Let $\varphi_\epsilon$ be a minimizer of $\mathcal{J}_\epsilon$, for any $\epsilon>0$. Then, there exists $\varphi \in \mathcal{U}$, minimizer of $\mathcal{J}_0$, such that, up to subsequences 
\begin{align*}
    \lim_{\epsilon \rightarrow 0}\|\varphi_{\epsilon} - \varphi\|_{L^1(S)}=0\quad\text{and}\quad\lim_{\epsilon \rightarrow 0}\mathcal{J}_{\epsilon}(\varphi_{\epsilon})=\mathcal{J}_0(\varphi).
\end{align*}
\end{corollary}
\begin{proof}
Since the sequence $(\varphi_\epsilon)_\epsilon$ is a sequence of minimizers, then
\begin{align*}
    \sup_{\epsilon>0}\mathcal{J}_\epsilon(\varphi_\epsilon)<\infty
\end{align*}
and, in particular,
\begin{align*}
    \sup_{\epsilon>0}\mathcal{E}_\epsilon(\varphi_\epsilon)<\infty.
\end{align*}

Then, arguing as in the proof of \cite[Proposition 3]{modica1987gradient}, the sequence of Ginzburg-Landau energies $(\mathcal{E}_\epsilon)_\epsilon$ is equicoercive in $L^{1}(S)$, and the thesis then follows as an application of Theorem \ref{thm:gamma_conv} and the fundamental theorem of $\Gamma$-convergence \cite[Theorem 7.8]{dal1993introduction}.
\end{proof}


\section{Numerical implementation}\label{numerics}


We assume in what follows that the reference cross-section $S$ is a simply connected domain.

To find minimizers for the rigidity optimization problem~\eqref{min:diff}, that involve the weighting factors $\gamma,\sigma_1,\sigma_2$ and $\sigma_3$ introduced in the previous sections, we use a steepest descent approach, as in~\cite[Appendix]{wolff2019twist}. 
It consists of computing a time-discrete $L^{2}$-gradient flow of the functional $J_{\epsilon}$, introduced in \eqref{Jeps}, until a stationary state has been reached.
We then discretize the domain $S$ by P1 triangular finite elements and apply a forward discretization in time, for time step $\tau$, and for integer iteration steps $n \geq 1$.
This leads to an artificial time variable $t=\tau \cdot n$, also called pseudo time.

 The implementation of the gradient flow method (C++-code) can be found in~\cite{stevegithub}.


\subsection{Gradient flow dynamics}\label{apx:gradflow} 

If we denote by $(\cdot,\cdot)$ the scalar product in the Hilbert space $L^{2}$, by $t$ the artificial time variable, and by $\frac{\delta J_\epsilon}{\delta \varphi}[v]$ the first variation of $J_\epsilon$ in $\varphi$ in the direction $v$, then the choice of a gradient flow dynamic leads to the pseudo-time stepping approach, given by 
\begin{equation}\label{eq:x2.1}
\begin{split}
    \epsilon(\partial_{t}\varphi,v)_{L^{2}(S)}&=-\frac{\delta J_{\epsilon}}{\delta \varphi}[v]\\
    &= -\left(\sigma_1\frac{\delta D_\mathrm{mean}}{\delta \varphi}[v] + \sigma_2\frac{\delta RM}{\delta \varphi}[v]+\sigma_3\frac{\delta D_T}{\delta \varphi}[v] + \gamma\frac{\delta E_{\epsilon}}{\delta \varphi}[v] \right)
\end{split}
\end{equation}
for any $\epsilon>0$ and for all $v \in W^{1,2}_0(S)$, where $D_T$, $D_\mathrm{mean}$, $RM$ and $E_{\epsilon}$ are respectively defined in \eqref{Dt}, \eqref{Dmean}, \eqref{RM} and \eqref{Eeps}. 
We note that the variation of the perimeter term produces the well-known Allen-Cahn equation.

To solve equation~\eqref{eq:x2.1} by a finite element method, we need to compute the first variations of $D_T$, $D_\mathrm{mean}$ and $RM$ in $\varphi$. 

 The variations of $D_\mathrm{mean}$ and $RM$ can be obtained by direct calculations, recalling that 
\begin{equation*}
\begin{split}
    \frac{\delta D_{x_2}}{\delta \varphi}[v]&=(1-c)\Bigg(\int\limits_{S} \hat{x}_2^{2} v ~ \mathrm{d}x_2\mathrm{d}x_3 -2 \int\limits_{S}(\varphi(1-c)+c) b_{1} \hat{x}_2~ \mathrm{d}x_2\mathrm{d}x_3\Bigg), \\
    \frac{\delta D_{x_3}}{\delta \varphi}[v]&=(1-c)\Bigg(\int\limits_{S} \hat{x}_2^{2}v ~ \mathrm{d}x_2\mathrm{d}x_3 -2 \int\limits_{S}(\varphi(1-c)+c) b_{2} \hat{x}_3~ \mathrm{d}x_2\mathrm{d}x_3\Bigg), \\
    \frac{\delta D_{x_2x_3}}{\delta \varphi}[v]&=(1-c)\Bigg(\int\limits_{S} \hat{x}_3 \hat{x}_2v ~ \mathrm{d}x_2\mathrm{d}x_3\\
    &\quad-\int\limits_{S}(\varphi(1-c)+c)( b_{1} \hat{x}_2+b_{2} \hat{x}_2) \mathrm{d}x_2\mathrm{d}x_3\Bigg),
\end{split}
\end{equation*}
where 
\[
b_{1}=\frac{1}{m}\int x_2 v ~ \mathrm{d}x_2\mathrm{d}x_3\quad\text{and}\quad b_{2}=\frac{1}{m}\int x_3 v ~ \mathrm{d}x_2\mathrm{d}x_2.
\]
Thus, we get  
\begin{align*}
\frac{\delta D_\mathrm{mean}}{\delta \varphi}[v]= \frac{\mu_{\rm norm}(3\lambda_{\rm norm}+ 2\mu_{\rm norm})}{\lambda_{\rm norm}+\mu_{\rm norm}}\left(\frac{\frac{\delta D_{x_2}}{\delta \varphi}[v]+\frac{\delta D_{x_3}}{\delta \varphi}[v]}{2}\right)
\end{align*} 
and 
\begin{align*}
    \frac{\delta RM}{\delta \varphi}[v]&=\frac{\mu_{\rm norm}(3\lambda_{\rm norm}+ 2\mu_{\rm norm})}{\lambda_{\rm norm}+\mu_{\rm norm}}\times\\
    &\quad\frac{(D_{x_2}-D_{x_3}) \left(\frac{\delta D_{x_2}}{\delta \varphi}[v]-\frac{\delta D_{x_3}}{\delta \varphi}[v]\right)+ 4D_{x_2x_3} \frac{\delta D_{x_2x_3}}{\delta \varphi}[v]}{2\sqrt{(D_{x_2}-D_{x_3})^{2}+4D_{x_2x_3}^{2}}}.    
\end{align*}
Regarding the regularization near the root $RM(\varphi)=0$, we introduce a parameter $0<\theta_{1}\ll 1$, and approximate $\frac{\delta RM} {\delta \varphi}[v]$ as
\[
\frac{\delta RM_{\theta_1}} {\delta \varphi}[v]= \frac{(D_{x_2}-D_{x_3})\left(\frac{\delta D_{x_2}}{\delta \varphi}[v]-\frac{\delta D_{x_3}}{\delta \varphi}[v]\right)+ 4D_{x_2x_3} \frac{\delta D_{x_2x_3}}{\delta \varphi}[v]}{2\sqrt{(D_{x_2}-D_{x_3})^{2}+4D_{x_2x_3}^{2}+\theta_{1}}}.
\]

For what concerns instead the first variation of $D_T$, we consider its distributional formulation~\eqref{eq:prandtl_phi} which, for simply connected cross-section $S$, reduces to 
\begin{equation}\label{eq:xn11.5}
    \int\limits_{S}\frac{1}{\mu_{\rm norm}(\varphi(1-c)+c)}\nabla \Phi \cdot \nabla v ~ \mathrm{d}x_2\mathrm{d}x_3 = 2\int\limits_{S}v ~ \mathrm{d}x_2\mathrm{d}x_3
\end{equation}
for all test functions $v\in W^{1,2}_0(S)$ and $\Phi=0$ on $\partial S$. 

 Since the variation of $D_T(\varphi)$ depends on the Prandtl's stress function $\Phi$ associated to $\varphi$ and given in equation~\eqref{eq:xn11.5},we apply a Lagrangian approach (see e.g. in~\cite{Hinze2009OptimizationWP}), by introducing the adjoint variable $p:S \rightarrow \mathbb{R}$.

We can then formulate the Lagrangian as
\begin{align*}
    L(\varphi,\Phi,p)&= E_{\epsilon}(\varphi)+ 2\int\limits_{S}\Phi ~ \mathrm{d}x_2\mathrm{d}x_3\\
    &\quad- \int\limits_{S}\frac{1}{\mu_{\rm norm}(\varphi(1-c)+c)}\nabla \Phi \cdot \nabla p ~ \mathrm{d}x_2\mathrm{d}x_3+2\int\limits_{S}p ~ \mathrm{d}x_2\mathrm{d}x_3.
\end{align*}
Looking for stationary states $(\varphi,\Phi,p)$ of $L$, we find that if the first variation for $(\varphi,\Phi,p)$ vanishes, both $\Phi$ and $p$ solve equation~\eqref{eq:xn11.5}.

 Moreover, since equation~\eqref{eq:xn11.5} is uniquely solvable, we can conclude that $\Phi=p$ and that
\begin{align*}
\frac{\delta L}{\delta \varphi}[v]=\frac{\delta E_{\epsilon}}{\delta \varphi}[v]+\int\limits_{S}\frac{1-c}{\mu_{\rm norm}(\varphi(1-c)+c)^{2}} \left(\nabla \Phi \cdot \nabla \phi\right) v ~ \mathrm{d}x_2\mathrm{d}x_3.
\end{align*}

This finally leads to the $L^2$-gradient flow 
\begin{align*}
    \epsilon(\partial_{t}\varphi,v)_{L^{2}(S)}= -\Bigg(\sigma_1\frac{\delta D_\mathrm{min}}{\delta \varphi}[v] + \sigma_2\frac{\delta D_\mathrm{max}}{\delta \varphi}[v] + \sigma_3\frac{\delta D_T}{\delta \varphi}[v] + \gamma\frac{\delta E_{\epsilon}}{\delta \varphi}[v] \Bigg),
\end{align*}
with 
\[
\frac{\delta D_T}{\delta \varphi}[v] = \int\limits_{S}\frac{1-c}{\mu_{\rm norm}(\varphi(1-c)+c)^{2}}\big(\nabla \Phi(\varphi) \cdot \nabla \Phi(\varphi)\big)v ~ \mathrm{d}x_2\mathrm{d}x_3
\]
and where $\Phi$ is the unique solution of~\eqref{eq:xn11.5}.


\subsection{P1-finite element approximation}\label{apx:fem}

In the following, let
\[
f(\varphi)=\frac{1}{2}(2\varphi^3-3\varphi^2+\varphi)
\]
denote the first derivative of the double well potential $F(\varphi)$ in~\eqref{double_well}, with respect to phase field variable $\varphi$.

In order to solve equation~\eqref{eq:x2.1} by the finite element method, we choose the discrete subspace $\mathcal{S}^{1,0}(\mathcal{T}_{h}) \subset W^{1,2}_0(S)$, which is given by
\[
\mathcal{S}^{1,0}(\mathcal{T}_{h})\coloneqq \{ v_{h}\in C^{0}(\bar{S}): (v_{h})_{|T}\in P^{1}(T),v_{h}=0\text{ on }\partial S \text{ for all }T \in\mathcal{T}_{h}\}.
\]
Furthermore, we choose a finite difference quotient to discretize the time derivative.

By using an explicit treatment of the appearing nonlinear terms in equation~\eqref{eq:x2.1}, and an implicit treatment of the Laplacian, we obtain the semi-implicit time stepping 
\begin{align*}
    &\int\limits_{S}\epsilon \frac{\mathcal{I}_{1,h}(\varphi^{n+1}-\varphi^{n})}{\tau} \phi_{j} ~ \mathrm{d}x_2\mathrm{d}x_3 \\
    &= -\int\limits_{S}\epsilon\gamma \nabla \mathcal{I}_{1,h}(\varphi^{n+1}) \cdot \nabla \phi_{j} + \frac{\gamma}{\epsilon} f(\mathcal{I}_{1,h}(\varphi^{n})) \phi_{j}~ \mathrm{d}x_2\mathrm{d}x_3\\ 
    &\quad- \sigma_{1}\frac{\delta D_{\mathrm{mean}}(\mathcal{I}_{1,h}(\varphi^{n}))}{\delta \phi_{j}} - \sigma_{2}\frac{\delta RM(\mathcal{I}_{1,h}(\varphi^{n}))}{\delta \phi_{j}} \\
    &\quad-\frac{\sigma_{3}}{\mu_{\rm norm}} \int\limits_{S}\frac{1-c}{(\mathcal{I}_{1,h}(\varphi^{n})(1-c)+c)^{2}} \nabla \Phi(\mathcal{I}_{1,h}(\varphi^{n})) \cdot \nabla \Phi(\mathcal{I}_{1,h}(\varphi^{n})) \phi_{j}\mathrm{d}x_2\mathrm{d}x_3, 
\end{align*}
with inner nodal basis points $x_{i}\in \mathcal{N}_{h}$ and $j =1,..., |\mathcal{N}_{h}|$.

 The nodal interpolant $\mathcal{I}_{1,h}(v)$ of a function $v\in C^{0}(\bar{S})$ is given by  
\[
\mathcal{I}_{1,h}v= \sum\limits_{j \in \mathcal{N}_{h}}\phi_{j}v(x_{j}),
\]
and the mass constraint is imposed by incorporating the additional condition
\[
\int\limits_{S}\mathcal{I}_{1,h}(\varphi^{n})~ \mathrm{d}x\mathrm{d}y-m=0.
\]
In every time step $t$ we thus solve the problem  
\begin{align}\label{eq:matrixsystem2}
\begin{pmatrix}
\hat{\mathcal{S}} & B\\
B^{T}& 0 
\end{pmatrix}\cdot U^{n+1} = \tilde{F},
\end{align}
where $\hat{\mathcal{S}}=M+\tau \gamma K \in \mathbb{R}^{n\times n}$ and $B=L$.
Here, the mass matrix $M$, the stiffness matrix $K$ and the diagonal lumped mass matrix $L \in \mathbb{R}^{n\times n}$ are, respectively, given by 
\begin{align*}
M_{ij} =\int\limits_{S} \phi_{j} \phi_{i} ~ \mathrm{d}x_2\mathrm{d}x_3, \ \   K_{ij}=\int\limits_{S}\nabla \phi_{j}\cdot \nabla \phi_{i} ~ \mathrm{d}x_2\mathrm{d}x_3, \ \ L_{ii}=\int\limits_{S}\phi_{i} ~ \mathrm{d}x_2\mathrm{d}x_3. 
\end{align*}
The right hand side $\tilde{F}$ in \eqref{eq:matrixsystem2} is determined by 
\begin{align*}
    \tilde{F}_{j}&=\frac{\tau}{\epsilon}F_{j}+(M\cdot U^{n})_{j}\quad\text{for }j=1,...,N,\\
    \tilde{F}_{N+1}&=m,
\end{align*}
where
\begin{align*}
    F_{j}\!&=\frac{\gamma}{\epsilon} \int\limits_{S}f(\mathcal{I}_{1,h}(\varphi^{n})) \phi_{j}~ \mathrm{d}x_2\mathrm{d}x_3\\
    &\!\!\quad - \sigma_{1}\frac{\delta D_{\mathrm{mean}}(\mathcal{I}_{1,h}(\varphi^{n}))}{\delta \varphi^{n}}[\phi_{j}] -\sigma_{2}\frac{\delta RM(\mathcal{I}_{1,h}(\varphi^{n}))}{\delta \varphi^{n}}[\phi_{j}] \\
    &\!\!\quad-\!\frac{\sigma_{3}}{\mu_{\rm norm}} \!\int\limits_{S}\!\frac{1-c}{(\mathcal{I}_{1,h}(\varphi^{n})(1-c)+c)^{2}} \nabla \Phi(\mathcal{I}_{1,h}(\varphi^{n})) \cdot \nabla \Phi(\mathcal{I}_{1,h}(\varphi^{n})) \phi_{j}\mathrm{d}x_2\mathrm{d}x_3.
\end{align*}

To achieve an energy stability in the way that for all $n\geq 1$ we have
\begin{align*}
    (\gamma E_{\epsilon}+\sigma_{1}D_{\mathrm{mean}}+\sigma_{2}RM+\sigma_{3}D_T)(\varphi^{n}) \leq\\
    (\gamma E_{\epsilon}+\sigma_{1}D_{\mathrm{mean}}+\sigma_{2}RM+\sigma_{3}D_T)(\varphi^{n-1}),
\end{align*}
it is necessary to demand $\tau \in \mathcal{O}(\epsilon^{3})$. In the following numerical simulations we set $\epsilon = 0.003$.


\subsection{Numerical experiments}

In this last part of the paper, inspired by multi-material composites found in the morphology of plant stems~\cite{rowe2004diversity,speck2020peak,wolff2021influence}, we study the optimal distribution of two materials inside a cross-section $S$, in the case in which the ratio between the Lam\'{e} parameters of the two materials is of order $\mathcal{O}(10^{-1})$.

In the following, this ratio is modeled within the density function $u^\varphi$, by using the phase field variable $\varphi$ and setting  
\[
u^\varphi(x_2,x_3)=\varphi(x_2,x_3)(1-c)+c
\]
with $c=0.1$, while the Lam\'{e} parameters $\mu_{\rm norm}$ and $\lambda_{\rm norm}$ are set to
\begin{align*}
    \begin{cases}
        \mu_{\rm norm}=26,\\
        \lambda_{\rm norm}=70.57.
    \end{cases}
\end{align*}

We consider different values of the weighting factors $\sigma_1,\sigma_2,\sigma_3$ and $\gamma$, that correspond to different terms of optimization, i.e. maximization or minimization of the torsional and bending rigidities.
The results of the numerical simulations are represented in Figs.~\ref{fig:trend},~\ref{fig:num_sim} and in Tab.~\ref{table:1}.

The reference cross-section $S$ is chosen as a circle of radius $r=0.7$ and the initial condition $\varphi^{0}$ corresponds to a circle of radius $r=0.5$.

 The discretization of $S$ is made up of approximately $1.6 \cdot 10^5$ P1 triangle elements.
The choice to start with a circular initial condition is justified by observations in the morphology of plant stems, where circular domains are often observed in young ontogenetic states, see for instance~\cite{rowe2004diversity,wolff2019twist}.

Since our goal is to optimize the overall bending rigidity of a rod with circular reference cross-section, in what follows we fix $\sigma_2=0$ and consider only the torsional rigidity $D_{T}$ \eqref{Dt} and the mean bending rigidity $D_{\rm mean}$ \eqref{Dmean}.
Depending on the weighting factors, we observe different stationary states of the gradient flow: in the case of a maximization of rigidities ($\sigma_1=-1$, $\sigma_3=-3$) and a sole minimization of torsional rigidity ($\sigma_3=3$), we obtain a symmetric circular tube and an I-beam like structure, see (a) and (b) in Fig.~\ref{fig:num_sim}.
These are well-known rigidity optimizers in the case of homogeneous rods (see, for instance,~\cite{kim2000topology}).

From the point of view of plant stem morphology, as noted in ~\cite{niklas1992plant,vogel2007living} and emphasised in the Introduction, plants are not inclined by a maximization of rigidity but more by an optimization of both strength and flexibility.
In general one observes a high twist-to-bend ratio
\[
\frac{D_{\rm mean}(u^\varphi)}{D_{T}(u^\varphi)},
\]
that implies high bending rigidity on the one hand and a comparatively low torsional rigidity on the other hand.

In our model, this can be achieved by a maximization of the bending rigidity and a minimization of the torsional rigidity ($\sigma_1=-3$, $\sigma_3=3$), and leads to a reinforcement by \textit{fibre strands}, which are formed by the stiffer material and are uniformly distributed along the boundary of $S$, see Fig.~\ref{fig:num_sim} (d), (e).

For a smaller weighting factor of the perimeter $\gamma$ a higher number of fibre strands occurs, see (e) in Fig.~\ref{fig:num_sim}.
This leads to a slight increase of the twist-to-bend ratio, which is mainly driven by a decrease of the torsional rigidity, see (c) in Fig.~\ref{fig:trend} and Tab.~\ref{table:1}.

Finally, we consider the minimization of both the bending an the torsional rigidities ($\sigma_1=1.5$, $\sigma_3=3$), that treats the case of achieving maximum flexibility.
This case is of particular interest for the study of vines, such as the liana in Fig.~\ref{fig:condylo}, with a trailing or shaking growth habit, where flexibility in both bending and torsion is crucial.

 Minimizing both the rigidities leads to structures with deep-grooves, see (f) and (g) in Fig~\ref{fig:num_sim}.
As in the experiments (d) and (e), a lower weighting of the perimeter functional has an influence on the shape of the minimizers.
In this case, a lower weighting of the perimeter term causes the stiffer material to form a deep groove shape with additional branched fingers, see (g) in Fig.~\ref{fig:num_sim}.


 The results in Fig.~\ref{fig:num_sim} (d) and (e) can be compared to the arrangement of fibre strands in several plant stems.
For instance, \textit{fibre-reinforced} structures, as in (d) and (e), are found in the morphology of \textit{Carex Pendula} and \textit{Caladium Bicolor}, where reinforcing by sclerenchymatous and collenchymatous stiffening tissues, respectively, results in a particularly high twist-to-bend ratio (see, for instance,~\cite{speck2020peak,wolff2022charting,wolff2021influence}).

On the other hand, the result in Fig.~\ref{fig:num_sim} (f) can be compared to the arrangement of a non-dense flexible secondary xylem with wide-diameter vessels and broad wood rays, as well as a flexible cortex during the ontogeny of a liana plant of the species \textit{Condylocarpon Guianense}, see Fig.~\ref{fig:condylo}.
During its ontogeny, this plant is inclined  to increasing its flexibility by a rearrangement of the main load bearing element, the secondary xylem. 
This is achieved by decreasing the bending and torsional rigidities of its plant stem by forming a structure with deep grooves and branched fingers.
For a detailed description of the evolution of the secondary xylem during the ontogeny of \textit{Condylocarpon Guianense} we refer the interested reader to~\cite{rowe2004diversity,wolff2019twist}. In particular, the transition from juvenile self-supporting liana stems with circular material arrangements to the non self-supporting adult plants with deep groves in the tissue arrangement is striking.

\begin{figure}[H]
     \centering
     \subfigure[$\sigma_1=-1,\sigma_3=-3$.]{\includegraphics[width=0.46\textwidth]{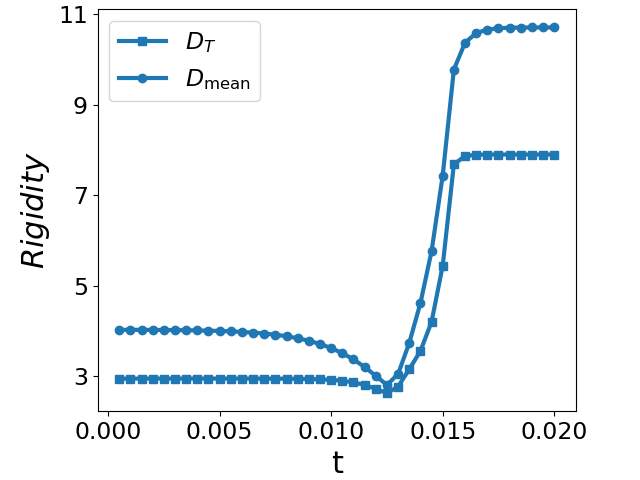}}\hspace{2em}
      \subfigure[$\sigma_1=0,\sigma_3=3$.]{\includegraphics[width=0.46\textwidth]{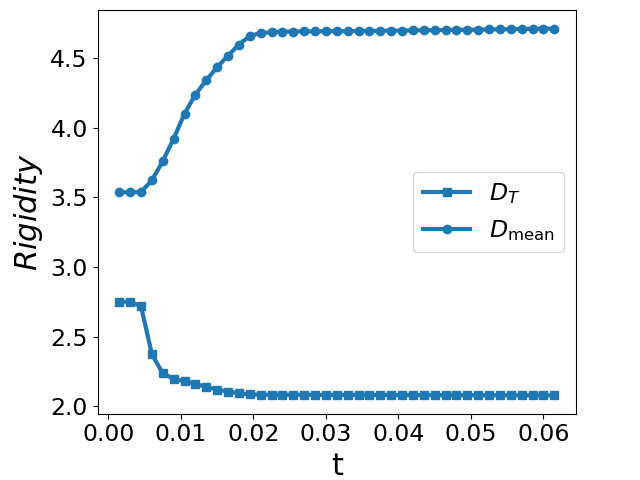}}\\
       \subfigure[$\sigma_1=-3,\sigma_3=3$.]{\includegraphics[width=0.46\textwidth]{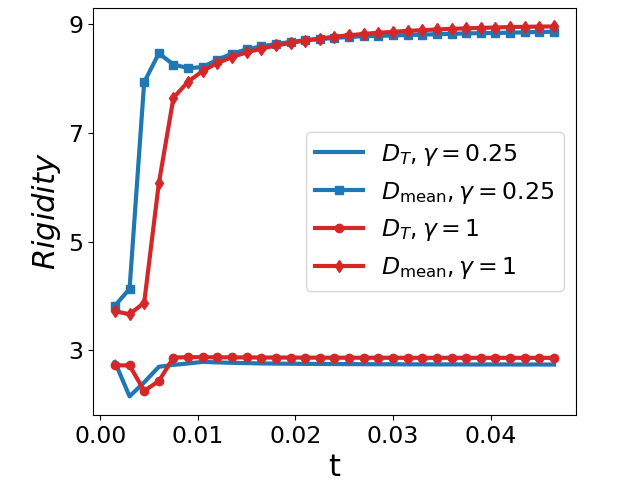}}\hspace{2em}
      \subfigure[$\sigma_1=1.5,\sigma_3=3$.]{\includegraphics[width=0.46\textwidth]{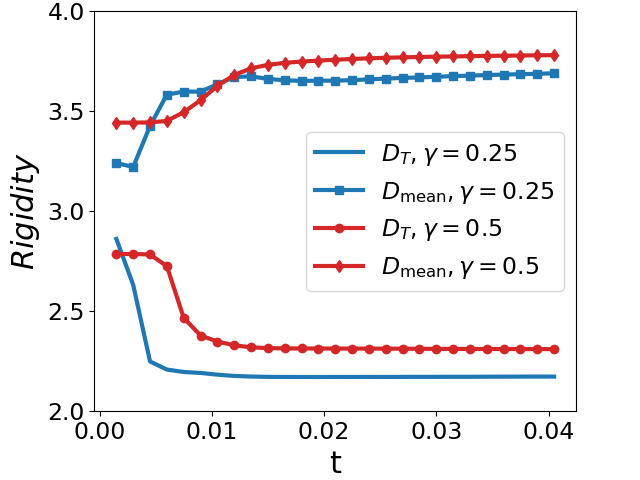}}\\
     \caption{Convergence history of the rigidities $D_{T}$ and $D_{\rm mean}$ with respect to pseudo-time $t$ during a gradient descent.}
     \label{fig:trend}
\end{figure}

\begin{table}[h!]
\captionsetup{font=footnotesize,labelfont=footnotesize}
\centering
\resizebox{0.75\textwidth}
{!}
{\begin{tabular}{|c | c | c |c|}
\hline
Experiment & $D_{\rm mean}(u)$ & $D_{T}(u)$ & $D_{\rm mean}(u)/D_{T}(u)$    \\ [1ex]
\hline
(a) & $3.54$ & $2.75$ & $1.29$ \\ [1ex]
\hline
(b) & $10.71$ & $7.9$ & $1.35$   \\ [1ex]
\hline
(c) & $4.71$ & $2.08$ & $2.26$ \\ [1ex] 
\hline
(d) & $8.96$ & $2.86$ & $3.13$   \\ [1ex]
\hline
(e) & $8.88$ & $2.74$ & $3.24$  \\ [1ex]
\hline
(f) & $3.79$ & $2.31$ & $1.64$   \\ [1ex]
\hline
(g) & $3.69$ & $2.17$ & $1.7$  \\ [1ex]
\hline
\end{tabular}}
\caption{Values of torsional and bending rigidities and the twist-to-bend ratio for stationary states in experiments (b)-(g) and initial condition (a).}
\label{table:1}
\end{table}

\begin{figure}[H]
\centering
\subfigure[Initial condition] {\includegraphics[width=0.26\textwidth]{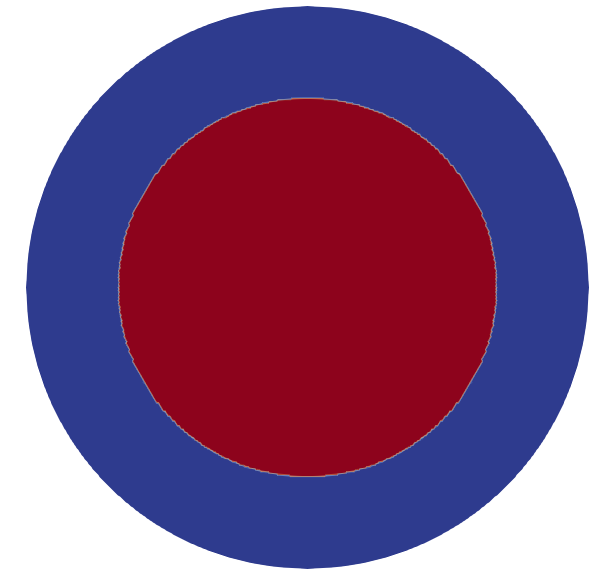}}\\
\subfigure[$\sigma_1=-1,\sigma_3=-3$.] {\includegraphics[width=0.26\textwidth]{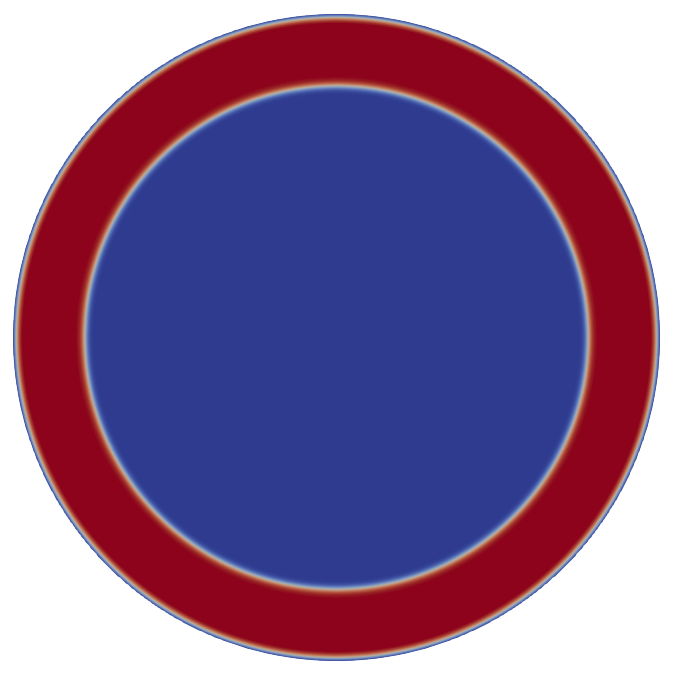}} \hspace{5em}
\subfigure[$\sigma_1=0,\sigma_3=3$.] {\includegraphics[width=0.26\textwidth]{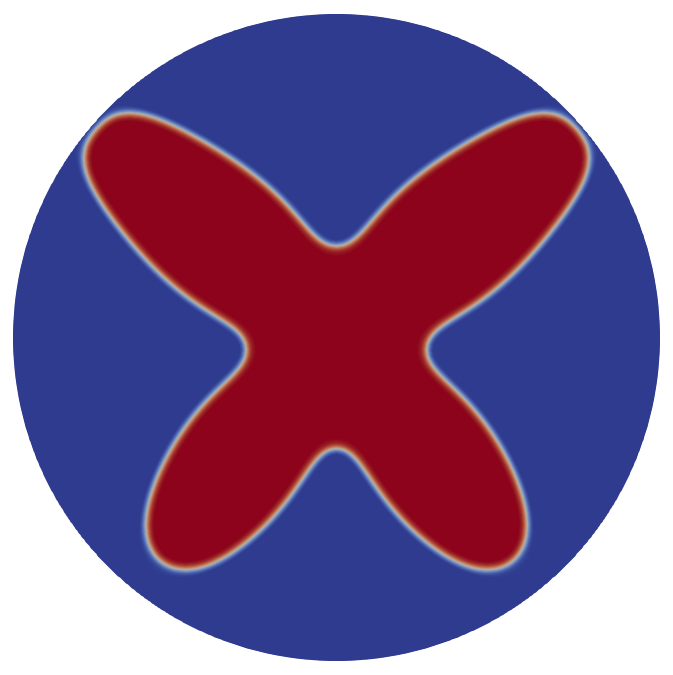}}\\
\subfigure[$\sigma_1=-3,\sigma_3=3$.] {\includegraphics[width=0.26\textwidth]{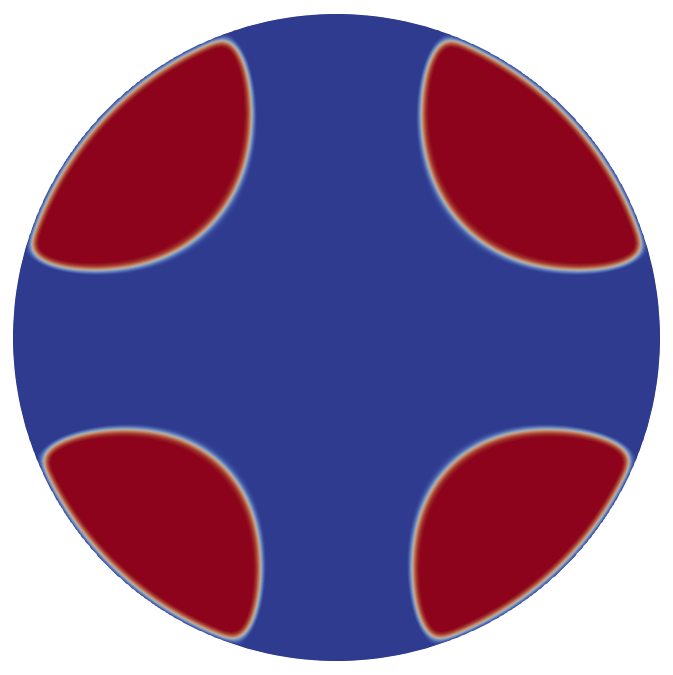}} \hspace{5em}
\subfigure[$\sigma_1=-3,\sigma_3=3$.] {\includegraphics[width=0.26\textwidth]{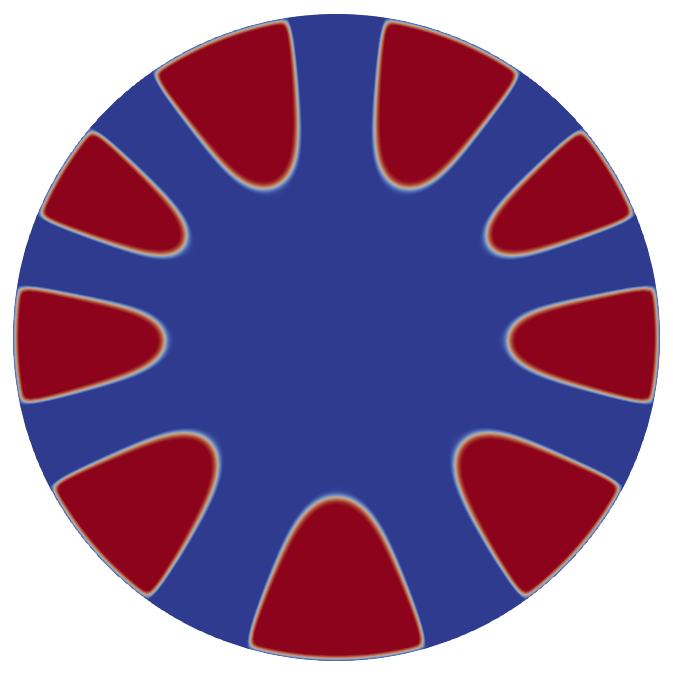}}\\
\subfigure[$\sigma_1=1.5,\sigma_3=3$.] {\includegraphics[width=0.26\textwidth]{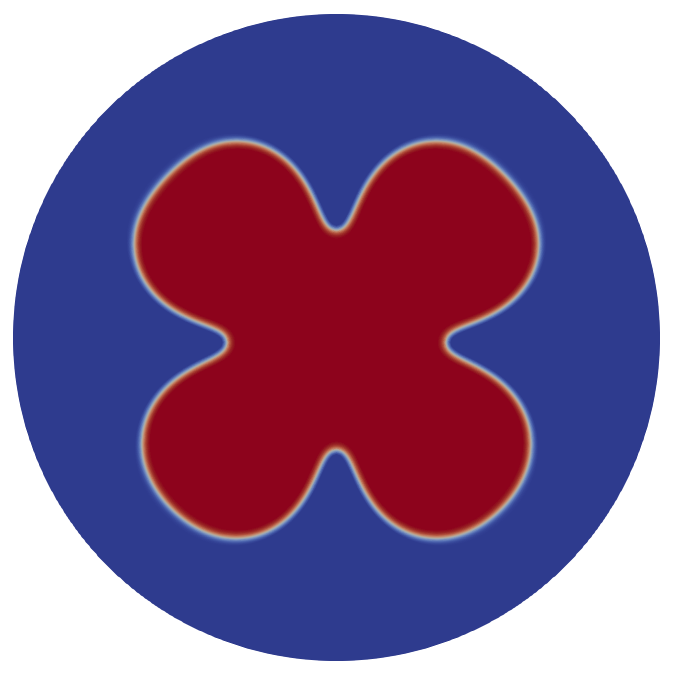}}\hspace{5em}
\subfigure[$\sigma_1=1.5,\sigma_3=3$.] {\includegraphics[width=0.26\textwidth]{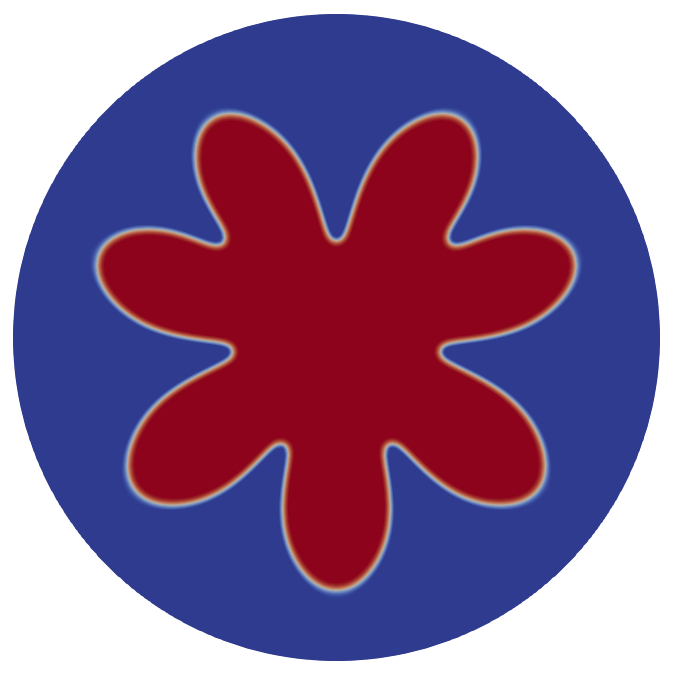}}
\caption[]{Local solutions of~\eqref{min:diff} for different weighting factors $\sigma_1,\sigma_3,\gamma$ corresponding to a maximization of rigidities in (b), a sole minimization of torsional rigidity in (c), a minimization of torsional and a maximization of bending rigidity in (d) and (e) and a minimization of both bending and torsional rigidity in (f) and (g). In experiments (b), (c) and (f) the weighting factor $\gamma$ for the perimeter penalization is set to $\gamma=0.5$, where in (e) and (g) the weighting factor is $\gamma=0.25$. In experiment (d) the weighting factor is set to $\gamma=1.0$. The stiffer material ($u=1$) is depicted in red where the softer material ($u=0.1$) is depicted in blue.}
\label{fig:num_sim}
\end{figure}


\section*{Conclusions}


Inspired by the bending-torsion theory of non-homogeneous elastic rods, we derived a model for the optimization of the bending and torsional rigidities.
It is done by studying a sharp interface shape optimization problem with perimeter penalization, that treats the torsional and bending rigidities as objectives.
We then adopted a diffuse interface approach, for which we proved the existence of solutions to the optimization problem.

 As a consequence of the $\Gamma$-convergence Theorem \ref{thm:gamma_conv}, we have showed that the sequence of minimizers of the diffuse interface approach converges to the minimizer of the sharp interface problem, with respect to the strong topology of $L^1(S)$, as the thickness of the interface tends to zero.

In the second part of the paper, we implemented a numerical method which approximates the solutions of the phase field problems, by using a steepest descent approach.
In particular, we studied four different cases of optimization: a maximization of both rigidities, a sole minimization of the torsional rigidity, a minimization of the torsional rigidity and a maximization of the bending rigidity, and, finally, the minimization of both rigidities.

The two latter cases were inspired by observations in plant morphology that showed that plants are more inclined to produce a high flexibility, especially in torsion, instead of a high stiffness of their stems. 
A numerical approximation of minimizers in these two cases resulted in characteristic shapes and distributions of two materials inside a circular cross-section $S$.

The appearing distributions of the materials exhibit a qualitative agreement with the tissue arrangements in the morphology of different plant stems. 
This study thus supports conjectures from biology stating that the optimization of certain mechanical properties, in particular the bending and torsional rigidities, are a driving force in the development of plant morphology \cite{Vogel1992,vogel2007living}.


\section*{Acknowledgments}


The authors would like to thank the livMatS Cluster of Excellence (Living, Adaptive and Energy-autonomous Materials Systems) and, in particular, the members Olga Speck and Thomas Speck, from the Plant Biomechanics Group Freiburg, for sharing their knowledge on Plant Biomechanics.~The authors would also like to thank Thomas Speck for providing the picture in Fig. \ref{fig:condylo}, Laura Melas for useful discussions on the topic and the anonymous referees for their careful reading.


\section*{Data Availability Statement}

The code used in this paper is available online in the Zenodo repository: \href{https://zenodo.org/records/10615223}{https://zenodo.org/records/10615223} DOI: 10.5281/zenodo.10615223 \cite{stevegithub}.


\bibliographystyle{abbrv} 
\bibliography{source}
\end{document}